\documentclass[a4paper,10pt]{article}

\usepackage[utf8x]{inputenc}
\usepackage[english]{babel}

\usepackage{amsmath}
\usepackage{graphicx}
\usepackage{hyperref}

\usepackage{amsfonts}
\usepackage{amsthm}
\usepackage{amssymb}
\usepackage{cite} 

\newcommand{\R}{\mathbb{R}} 
\newcommand{\N}{\mathbb{N}} 
 
\newcommand{\Z}{\mathbb{Z}} 

\newcommand{\graph}{\operatorname{graph}}
\newcommand{\supp}{\operatorname{spt}}

\newtheorem{theorem}{Theorem}[section]
\newtheorem{lemma}[theorem]{Lemma}

\theoremstyle{definition}

\newtheorem{remark}[theorem]{Remark}

\numberwithin{equation}{section}

\title{Lower-semicontinuity for the Helfrich problem.}

\author{Sascha Eichmann\thanks{The author thanks Prof. Reiner Sch\"atzle for discussing the Helfrich energy and providing insight into geometric measure theory.}\\
Mathematisch-Naturwissenschaftliche Fakultät,\\
Eberhard Karls Universität Tübingen,\\
Auf der Morgenstelle 10,\\
D-72076 Tübingen, Germany\\
E-mail: \href{mailto:sascha.eichmann@math.uni-tuebingen.de}{sascha.eichmann@math.uni-tuebingen.de}\\
Phone: +49/7071/2976886}

\begin{document}
 \maketitle

\begin{abstract}
We minimise the Canham-Helfrich energy in the class of closed immersions with prescribed genus, surface area and enclosed volume.\\
Compactness is achieved in the class of oriented varifolds. The main result is a lower-semicontinuity estimate for the minimising sequence,
which is in general false by a counter example by Gro\ss e-Brauckmann.
The main argument involved is showing partial regularity of the limit. 
It entails comparing the Helfrich energy of the minimising sequence locally to that of a biharmonic graph. 
This idea is by Simon, but it cannot be directly applied, since the area and enclosed volume of the graph may differ.
By an idea of Schygulla we adjust these quantities by using a two parameter diffeomorphism of $\R^3$. 

\end{abstract}
\textbf{Keywords.} Closed immersions, Canham-Helfrich energy, lower-semicontinuity, variational problem, oriented varifolds\\ 
\textbf{MSC.}  26B15, 35J35, 49J45, 49Q20, 58J90

\section{Introduction}
\label{sec:1}

This article deals with minimising the Helfrich energy for closed oriented smooth connected two dimensional immersions $f:\Sigma\rightarrow\R^3$, which is defined as follows
\begin{equation}
\label{eq:1_1}
W_{H_0}(f):=\int_\Sigma (\bar{H}-H_0)^2\, d\mu_g. 
\end{equation}
Here $\bar{H}$ is the scalar mean curvature, i.e. the sum of the principal curvatures with respect to a choosen unit normal of $f$. 
$\mu_g$ is the area measure on $\Sigma$ induced by $f$ and the euclidean metric of $\R^3$.
$H_0\in\R$ is called spontaneous curvature. 
This energy was introduced by Helfrich in \cite{Helfrich} and Canham \cite{Canham} to modell the shape of blood cells.
Hence it is called the Canham-Helfrich or short Helfrich energy.
We recover the Willmore energy by setting $H_0=0$ and multiplying by $\frac{1}{4}$, i.e.
\begin{equation}
\label{eq:1_2}
 W_{Will}(f):=\frac{1}{4}\int_\Sigma |\bar{H}|^2dA = \frac{1}{4}W_{0}(f).
\end{equation}
The Willmore energy goes back to Thomsen \cite{Thomsen} and he denoted our modern Willmore surfaces as conformal minimal surfaces.
Willmore later revived the mathematical discussion in \cite{Willmore}. 
Please note, that in the Willmore setting an orientation is not needed, contrary to the Helfrich energy,
hence we need a fixated normal. We assume $\Sigma$ to have a continuous orientation $\tau$.
Then we set $\nu_f:=*(df(\tau))\in \partial B_1(0)\subset \R^3$ as the unit normal of $f$. Here $*$ denotes the Hodge-$*$-Operator. 
Furthermore we like to prescribe the area and enclosed volume of $f$. Therefore we set
\begin{equation}
 \label{eq:1_3}
 Area(f):=\int_\Sigma d\mu_g,\quad Vol(f)=\frac{1}{3}\int_\Sigma f(x)\cdot \nu_f(x)\, d\mu_g(x).
\end{equation} 
Please note, that if $f$ is an embedding and $\nu_f$ the outer normal, $Vol(f)$ would be the volume of the set enclosed by $f$.
In the general case $Vol(f)$ may become negativ dependend on the orientation.\\
Minimisers of such a problem satisfy the following Euler-Lagrange equation 
(see e.g. \cite[Eq. (31)]{OuYang_Helfrich})
\begin{equation}
 \label{eq:1_4}
 2\Delta_f \bar{H} + 4\bar{H}\left(\frac{1}{4}\bar{H}^2-K\right)-2H_0 K -H_0^2\bar{H}-\lambda_A \bar{H} + \lambda_V=0.
\end{equation}
Here $\lambda_A\in\R$ and $\lambda_V\in\R$ correspond to Lagrange multipliers for the prescribed area and enclosed volume respectively.
This differential equation is highly nonlinear since the Laplace-Beltrami $\Delta_f$ depends on the unknown immersion $f$.
Furthermore it is of fourth order, hence standard techniques like the maximum principle are not applicable.
Nevertheless a lot of important results concerning such problems have been achieved:
Existence of closed Willmore surfaces of arbitrary genus has been shown in the papers \cite{Simon} and \cite{BauerKuwert}.
If the ambient space becomes a general $3$-manifold, the problem has been examined in \cite{MondinoSchygulla}.\\
Prescribing additional conditions and showing existence has also been very successful, i.e.
the isoperimetric ratio in \cite{KellerMondinoRiviere} and \cite{Schygulla}, the area  in \cite{LammMetzger}, 
boundary conditions in \cite{Schaetzle}, \cite{Pozetta}, \cite{DaLioPalmurellaRiviere} and \cite{DeckGruRoe}.
Finally the Willmore conjecture was shown to be true in \cite{NevesMarques}.
Further research and references for Willmore problems are summarized in the surveys \cite{HellerPedit} rsp. \cite{GrunauWillmoreSurvey}.
The class of axisymmetric surfaces is especially important for modelling purposes, see e.g. \cite{DeulingHelfrich}.
Also some existence results can be more readily achieved in this class, see e.g. \cite{Doemeland}, \cite{DallDeckGru} or \cite{EichmannGrunau}
and solutions to \eqref{eq:1_4} can be analysed in greater depths, 
e.g. the behaviour of singularities in \cite{FourcadeZia}.

Furthermore the addition of the orientation complicates and/or changes the situation in the Helfrich setting. 
For example the class of invariances is notably smaller (see e.g. \cite{DeMatteisManno}) and lower semi-continuity is in general false under varifold convergence, 
see \cite[p. 550, Remark (ii)]{GrosseBrauck} for a counterexample.
Nevertheless some progress has been made:
Existence of Helfrich surfaces with prescribed surface area near the sphere has been achieved in \cite{KohsakaNagasawa} by examining the corresponding $L^2$-flow.
This result has been extendend to a general existence result for spherical Helfrich immersions with prescribed area and enclosed volume in \cite{MondinoScharrer}
by variational parametric methods. The axisymmetric case has been handled independently in \cite{ChoksiVeroni} and \cite{ChoksiMorandottiVeneroni}.
A more general approach was used by Delladio by working with Gauss-graphs in \cite{Delladio}.
There lower-semicontinuity was shown, but the limit has to be in $C^2$, which is a-priori not clear.

In general this is a hard problem, since the Helfrich energy for oriented varifolds lacks a variational characterisation (cf. \eqref{eq:2_3} and cf. \cite[p. 3]{EichmannHelfrichBoundary}).

In this paper we will show a lower-semicontinuity estimate for minimising sequences with an arbitrary but fixated topology by an ambient approach, i.e. geometric measure theory.
In the case of spherical topology this has been achieved with a parametric approach in \cite[Thm. 3.3]{MondinoScharrer}. 
The case of embeddings still remains open, since we do not have a Li-Yau type inequality. 
This prevents our arguments to be adaptable to this situation as explained in Remark \ref{5_2}.

For our minimising procedure we introduce the following set for a given two dimensional smooth connected manifold $\Sigma$ without boundary, which fixates the topology.
\begin{equation}
\label{eq:1_5}
 M_{Area_0,Vol_0}:=\left\{\begin{array}{c}f:\Sigma\rightarrow\R^3,\ 2\mbox{-dimensional oriented smooth immersion,}\\ Area(\Sigma)=Area_0,\ Vol(\Sigma)=Vol_0,\ 0\in f(\Sigma)\end{array}\right\}
\end{equation}
for a positive parameter $Area_0>0$ and a nontrivial real parameter $Vol_0\in\R\setminus\{0\}$.
Furthermore let us call
\begin{equation}
 \label{eq:1_6}
 E_{H_0,Area_0,Vol_0}:=\inf\{W_{H_0}(f),\ f\in M_{Area_0,Vol_0}\}.
\end{equation}
Now we can state our main result in case $M_{Area_0,Vol_0}\neq\emptyset$ (For unknown terminology please consult section \ref{sec:2}):
\begin{theorem}
 \label{1_1}
 Let $f_k\in  M_{Area_0,Vol_0}$ be a minimising sequence, i.e. it satisfies $W_{H_0}(f_k)\rightarrow E_{H_0,Area_0,Vol_0}$.
 Let also $V^0_k$ be the sequence of oriented integral $2$ varifolds corresponding to $f_k$ by \eqref{eq:2_3_1}.
 Then there exists an oriented integral $2$ varifold $V^0$ on $\R^3$ and a subsequence $V^0_{k_j}$, such that
 \begin{align*}
  &V_{k_j}\rightarrow V^0\mbox{ as oriented varifolds,}\\
  &Area(V^0)=Area_0,\quad Vol(V^0)=Vol_0\neq 0.
 \end{align*}
More importantly this subsequence enjoys a lower-semicontinuity estimate:
\begin{equation*}
 W_{H_0}(V^0) \leq \liminf_{j\rightarrow\infty} V^0_{k_j}.
\end{equation*}
Furthermore there exist at most finitely many bad points, such that $V^0$ is locally a union of $C^{1,\beta}\cap W^{2,2}$ graphs outside these bad points.
For a precise statement of the graphical decomposition please refer to Lemma \ref{5_1}.
\end{theorem}

The proof works in two major steps. The first is to show some preliminary regularity of $V^0$ (see Lemma \ref{5_1}).
After this the lower-semicontinuity estimate follows precisely as in \cite[Lemma 4.1]{EichmannHelfrichBoundary}, hence we will not include this step here.

The initial regularity is shown by comparing the Helfrich energy of the minimising sequence to a biharmonic replacement. This idea was first used by Simon in \cite{Simon}.
Since the biharmonic replacement does not have the same area and/or enclosed volume, we will correct these parameters by an idea of Schygulla \cite{Schygulla}.
We generalise Schygulla's techniques to the case of immersions and to two prescribed quantities. 
This idea essentially is adjusting these quantities of the biharmonic replacement outside of the replacement region by a two parameter diffeomorphism.

The paper is build up as follows: Section \ref{sec:2} is concerned with compactness in the class of oriented varifolds.
Here the Helfrich energy, the area and the enclosed volume are formulated for these varifolds.
Next in section \ref{sec:3} we examine our definition of enclosed volume for oriented varifolds in greater detail and for example calculate an Euler-Lagrange equation.
Afterwards we construct the aforementioned two parameter diffeomorphism and analyse these in \ref{sec:4}.
In section \ref{sec:5} we finally show the initial $C^{1,\beta}$-regularity, from which the lower-semicontinuity estimate follows.
In Appendix \ref{sec:A} we collect some usefull results for our reasoning.

\section{Compactness}
\label{sec:2}
In this chapter we recall the necessary results and objects to obtain compactness for a minimising sequence in measure theoretic terms:

Since the Helfrich energy is dependend on the choosen orientation, we will have to work with oriented varifolds in our variational framework.
Oriented varifolds were introduced by Hutchinson in \cite[§ 3]{Hutchinson}. 
We recall the necessary definitions here (see also \cite[Appendix B]{EichmannHelfrichBoundary}). First
\begin{equation*}
G^0(2,3)=\{\tau^1\wedge\tau^2\in\Lambda_2\R^{3}:\ \tau^1,\tau^2\in\R^3,\ |\tau^1|=|\tau^2|=1,\ \tau^1\perp\tau^2\}
\end{equation*}
is called the oriented Grassmannian manifold of $2$-dimensional oriented linear subspaces in $\R^3$. 
Since we need to connect orientations with normals, we also need the Hodge star operator $*:G^0(2,3)\rightarrow \partial B_1(0)$.
In our setting $*$ becomes just the cross product, i.e. $*(\tau^1\wedge\tau^2)=\tau^1\times \tau^2$.
An oriented integral $2$-varifold on an open set $\Omega\subset\R^3$ is given by a countable $2$-rectifiable set $M\subset O$, 
$\mathcal{H}^2$-measurable densities $\theta_+,\theta_-:M\rightarrow \N_0$ and an orientation
$\xi:M\rightarrow G^0(2,3)$, such that $*(\xi(x))\perp T_x M$ for $\mathcal{H}^2$ a.e. $x\in M$. 
Then the corresponding oriented varifold is a Radon measure on $\Omega\times G^0(2,3)$ given for $\Phi\in C^0_0(\Omega\times G^0(2,3))$ by
\begin{equation}
 \label{eq:2_1}
 V^0[M,\theta_+,\theta_-,\xi](\Phi):=\int_M \Phi(x,\xi(x))\theta_+(x) + \Phi(x,-\xi(x))\theta_-\, d\mathcal{H}^2(x).
\end{equation}
Furthermore we need to define the Helfrich energy for such an oriented integral varifold. 
Hence we need to make sense of a mean curvature. 
For this let $\pi^0: \R^3\times G^0(2,3)$ be given by $\pi^0(x,\xi)=x$. Then the mass of $V^0$ is defined as
\begin{equation*}
 \mu_{V^0}:=\pi^0(V^0) = (\theta_++\theta_-)\mathcal{H}^2\lfloor M,
\end{equation*}
which is also an integral varifold in the sense of \cite[§ 15]{Simon_Buch}.
The first variation of $V^0$ is defined as the first variation of $\mu_{V^0}$, i.e. $\delta V^0:=\delta \mu_{V^0}$ 
Thus we define the mean curvature vector of $V^0$ to be the mean curvature vector of $\mu_{V^0}$ (cf. \cite[§16]{Simon_Buch}).
If we say $V^0$ has a mean curvature vector $H_{V^0}\in L^2(\mu_{V^0})$, then it is square integrable w.r.t. $\mu_{V^0}$ and for every $X\in C^1_0(\Omega,\R^3)$ we have
\begin{equation*}
 \delta V^0(X)= -\int_\Omega H_{V^0}\cdot X\, d\mu_{V^0}.
\end{equation*}
In the sense of \cite[§ 39]{Simon_Buch} this means that $\mu_{V^0}$ does not have a generalized boundary.

Furthermore we can define an integral $2$ current associated to $V^0$ by
\begin{equation*}
 [|V^0|](\omega):=\int_{G^0(\Omega)}\langle \omega(x), \xi \rangle\, dV^0(x,\xi),\ \omega\in C_0^\infty(\Omega,\Lambda^2\R^{3}).
\end{equation*}
We choose the same notations for currents as \cite[Chapter 6]{Simon_Buch}.

For an integral oriented varifolds $V^0=V^0[M,\theta_+,\theta_-,\xi]$ the current also satisfies
\begin{equation*}
 [|V^0|](\omega) = \int_M \langle \omega(x),\xi(x)\rangle (\theta_+(x)-\theta_-(x))\, d\mathcal{H}^2
\end{equation*}
and is therefore integral as well.
We will call $\partial[|V^0|]$ the boundary in the sense of currents of $V^0$.
Here $\partial$ is the boundary operator for currents (see e.g. \cite[Eq. (26.3)]{Simon_Buch}). 

Convergence of oriented varifolds is defined as weak convergence of the corresponding Radon measures. 
I.e. we say a sequence of oriented integral varifolds $V^0_k$ on $\Omega\subset \R^3$ converges weakly to an oriented integral varifold $V^0$,
iff for all $\Phi\in C^0_0(\Omega\times G^0(2,3))$ we have $V^0_k(\Phi)\rightarrow V^0(\Phi)$.
\cite[Theorem 3.1]{Hutchinson} now gives us 
that the following set is sequentially compact with respect to oriented varifold convergence:
\begin{align}
\label{eq:2_2}
\begin{split}
 \{V^0\mbox{ oriented integral varifold}:\ \forall \Omega'\subset\subset\Omega\ \exists C(\Omega')<\infty:&\\ \mu_{V^0}(\Omega')+\|\delta\mu_{V^0}\|(\Omega') + M_{\Omega'}(\partial[|V^0|])\leq C(\Omega')\}
\end{split}
 \end{align} 
 Here $M_{\Omega'}(\cdot)$ denotes the mass in the sense of currents.

 The Helfrich energy of $V^0$ is (see also \cite[Eq. (2.1)]{EichmannHelfrichBoundary})
\begin{align}
 \label{eq:2_3}
 \begin{split}
 W_{H_0}(V^0)=&\int_M |H_{V^0}(x)-H_0(*\xi(x))|^2\theta_+(x)\\& + |H_{V^0}(x)+H_0(*\xi(x))|^2\theta_-(x)\, d\mathcal{H}^2(x)\\
 =& \int |H_{V^0}(x)-H_0(*\xi)|^2\, dV^0(x,\xi).
\end{split}
 \end{align}

Furthermore we need to define the enclosed volume and the area of such oriented varifolds:
\begin{equation*}
 Area(V^0):=\mu_{V^0}(\R^3),\quad Vol(V^0):=\frac{1}{3}\int \langle x, (*\xi)\rangle\, dV^0(x,\xi). 
\end{equation*}

Now let $\Sigma$ be a smooth oriented $2$-dimensional manifold and $f:\Sigma\rightarrow\R^3$ a smooth immersion.
To employ the compactness criterion \eqref{eq:2_2} we need to define a corresponding oriented integral varifold (see \cite[§2]{EichmannHelfrichBoundary}):
Let $\xi_f:f(\Sigma)\rightarrow G^0(2,3)$ be an $\mathcal{H}^2$ measurable orientation. Then the corresponding oriented integral varifold is
\begin{equation}
\label{eq:2_3_1}
 V^0_f:=V^0(f(\Sigma),\theta_+,\theta_-,\xi_f).
\end{equation}
To define the densities let us denote the choosen continuous orientation of $T_x\Sigma$ by $\tau(x)$. 
Then $\theta_+,\theta_-:f(\Sigma)\rightarrow \N_0$ are defined by
\begin{align}
\begin{split}
\label{eq:2_4}
 \theta_+(y)&=\sum_{x\in f^{-1}(y)}\operatorname{sign}_+(df(\tau(x))\diagup (\xi_f(y)),\\ \theta_-(y)&=\sum_{x\in f^{-1}(y)}\operatorname{sign}_-(df(\tau(x))\diagup (\xi_f(y)).
\end{split}
 \end{align}
Here 
\begin{equation*}
 \operatorname{sign}_+(df(\tau(x))\diagup (\xi_f(y))=\left\{\begin{array}{cc}1,&\mbox{ if }\xi_f(y)\mbox{ is the same orientation as }df(\tau(x))\\ 0,&\mbox{ else.}\end{array}\right. 
\end{equation*}
Analogously $\operatorname{sign}_-(df(\tau(x))\diagup (\xi_f(y))=1$, if $\xi_f(y)$ is the opposite orientation of $df(\tau(x))$. 
Please note, that these densities are only well defined $\mathcal{H}^2\lfloor f(\Sigma)$ almost everywhere, which is enough to obtain a well defined oriented varifold.

Let $f_k\in  M_{Area_0,Vol_0}$ be a minimising sequence as in Theorem \ref{1_1} with orientation $\xi_k:f_k(\Sigma)\rightarrow G^0(2,3)$ and no boundary, i.e. closed.
Furthermore let
\begin{equation}
 \label{eq:2_5}
 V^0_k:=V^0(f_k(\Sigma),\theta^k_+,\theta^k_-,\xi_k),
\end{equation}
with $\theta_\pm^k$ defined as in \eqref{eq:2_4}.
Let $H_k:f_k(\Sigma)\rightarrow\R^3$ be the mean curvature vector of $V^0_k$. Then the Helfrich energy of $f_k$ becomes
\begin{equation*}
 W_{H_0}(f_k)=\int_{f_k(\Sigma)} |H_k(x)-H_0(* \xi_k(x))|^2\theta_+^k + |H_k(x)+H_0(* \xi_k(x))|^2\theta_-^k \,d\mathcal{H}^2
\end{equation*}

By the Cauchy-Schwartz's and Young's inequality we get for $\varepsilon>0$ arbitrary
\begin{align*}
 &\int_{f_k(\Sigma)} |H_k(x)-H_0(* \xi_k(x))|^2\theta_+^k\,d\mathcal{H}^2\\
 =& \int_{f_k(\Sigma)} \left(|H_k|^2 - 2H_0 \langle H_k,(*\xi_k)\rangle + H_0^2\right))\theta_+^k\,d\mathcal{H}^2\\
 \geq& \int_{f_k(\Sigma)} \left((|H_k|^2 - 2|H_0||H_k| + H_0^2\right)\theta_+^k\,d\mathcal{H}^2\\
 \geq& \int_{f_k(\Sigma)} \left(|H_k|^2 - \varepsilon|H_k|^2 - \frac{1}{\varepsilon}H_0^2 + H_0^2\right)\theta_+^k\,d\mathcal{H}^2\\
 =& (1-\varepsilon)\int_{f_k(\Sigma)} |H_k|^2\theta_+^k\,d\mathcal{H}^2 + \left(1-\frac{1}{\varepsilon}\right)H_0^2\int_{f_k(\Sigma)}\theta_+^k\, d\mathcal{H}^2.
\end{align*}
By repeating the argument with $\theta_-^k$, we obtain by choosing $\varepsilon\in(0,1)$
\begin{equation}
 \label{eq:2_6}
\frac{W_{H_0}(f_k) + (\frac{1}{\varepsilon}-1)H_0^2Area_0}{1-\varepsilon}\geq \int_{\Sigma_k}|H_k|^2\,(\theta_+^k+\theta_-^k)d\mathcal{H}^2 = \int |H_k|^2\, d\mu_k.
\end{equation}
and hence a $C=C(H_0)>0$, such that
\begin{equation}
 \label{eq:2_6_1}
 \int |H_k|^2\, d\mu_k\leq C \left( W_{H_0}(f_k) + Area(f_k)\right).
\end{equation}


By the definition of $M_{Area_0,Vol_0}$ we also have
\begin{equation}
\label{eq:2_7_1}
 0\in f_k(\Sigma).
\end{equation}
for every $k\in\N$.
Since the $f_k$ are closed we have $\partial [|V^0_k|]=0$. 
Hence equation \eqref{eq:2_5} allows us to employ Hutchinsons compactness result for oriented integral varifolds \cite[Thm. 3.1]{Hutchinson} rsp. \eqref{eq:2_2} and obtain
an oriented integral $2$-varifold $V^0=V^0[M,\theta_+,\theta_-,\xi]$, such that after extracting a subsequence and relabeling we have
\begin{equation*}
 V^0_k\rightarrow V^0\mbox{ as oriented varifolds.}
\end{equation*}
Furthermore let $A_k$ be the second fundamental form of $f_k$. 
By \eqref{eq:2_5} and by only dealing with fixated topology for $f_k$ we get the following estimate
\begin{equation*}
 \int_{f_k(\Sigma)}|A_k|^2\,d\mathcal{H}^2\leq C,
\end{equation*}
for some constant $C>0$ independent of $k$ (see also \cite[Eq. (1.1)]{Schaetzle}).
Hence by \cite[Thm. 5.3.2]{Hutchinson} $\mu:=\mu_{V^0}$ has a weak second fundamental form $A_\mu\in L^2(\mu)$. 
By possibly extracting another subsequence we obtain 
\begin{align}
 \label{eq:2_8}
 \begin{split}
   V^0_k\rightarrow V^0 &\mbox{ as oriented varifolds,}\\
   \mu_k\rightarrow \mu &\mbox{ weakly as varifolds,}\\
   |A_k|^2\mu\rightarrow \nu & \mbox{ weakly as Radon measures,}\\
   |A_\mu|^2\mu \leq \nu &\mbox{ and } \nu(\R^3)\leq C<\infty.
 \end{split}
\end{align}
Without loss of generality we also have $M=\operatorname{spt}(\mu)$.\\
Before we can proceed we need to ensure that the limit $V^0$ satisfies $Area(V^0)=A_0$ and $Vol(V^0)=Vol_0$.
For the first one, we use the following lemma
\begin{lemma}[see \cite{EichmannHelfrichBoundary}, Lemma 2.1]
 \label{2_1}
 $\operatorname{spt}(\mu)$ is compact.
\end{lemma}
\begin{proof}
The proof is analouge to \cite[Lemma 2.1]{EichmannHelfrichBoundary}.
For the reader's convenience we include it here.
  Simon's diameter estimate \cite[Lemma 1.1]{Simon} and the bound on the Willmore energy yield
 \begin{equation}
 \label{eq:2_9}
  \operatorname{diam}(f_k(\Sigma))\leq C\sqrt{\mu_{k}(\R^3)\cdot W_{Will}(f_k)}\leq C.
 \end{equation}
Now let $x\in \operatorname{spt}(\mu)$. For an arbitrary $\rho>0$ we obtain by e.g. \cite[Prop. 4.26]{Maggi} and the defintion of the support of a Radon measure
\begin{equation*}
 0<(\mu)(B_\rho(x))\leq \liminf_{k\rightarrow\infty}(\mu_k)(B_\rho(x)).
\end{equation*}
Hence $spt(\mu_k)\cap B_\rho(x)\neq\emptyset$ for $k$ big enough. Therefore we can find $x_k\in spt(\mu_k)$ such that $x_k\rightarrow x$. By \eqref{eq:2_9} we finally obtain 
\begin{equation*}
 \operatorname{diam}(\operatorname{spt}(\mu))\leq C
\end{equation*}
and the lemma is proven.
\end{proof}
\eqref{eq:2_9} and \eqref{eq:2_7_1} yield a constant $N>0$, such that for all $k$ we have 
\begin{equation}
 \label{eq:2_10}
 \operatorname{spt}(\mu),f_k(\Sigma)\subset B_N(0).
\end{equation}
Now choose a smooth cut-off function $\varphi\in C_0^\infty(B_{2N}(0))$, such that $\varphi=1$ on $B_N(0)$.
The varifold convergence of $\mu_k$ now yields 
\begin{equation*}
Area_0=\lim_{k\rightarrow\infty}\mu_k(\R^3)= \lim_{k\rightarrow \infty}\int \varphi\, d\mu_k=\int\varphi\, d\mu=\mu(\R^3).
\end{equation*}
Also $(x,\xi)\mapsto \varphi(x)\langle x,*(\xi)\rangle$ defines a continuous function with compact support on $\R^3\times G^0(2,3)$. 
Hence the oriented varifold convergence yields
\begin{align*}
 &Vol_0=Vol(f_k) = \frac{1}{3}\int \varphi(x)\langle x,*(\xi)\rangle \, dV^0_k(x,\xi)\\ \rightarrow\ & \frac{1}{3}\int \varphi(x)\langle x,*(\xi)\rangle\, dV^0(x,\xi) = Vol(V^0).
\end{align*} 
The enclosed volume will need more attention, since we have to calculate a first variation.
We will do this in section \ref{sec:3}

\section{First Variation of enclosed volume}
\label{sec:3}
In this section we like to derive a suitable formula for the first variation of the enclosed volume with respect to a smooth vectorfield.
Let $V^0=V^0[M,\theta_+,\theta_-,\xi]$ be an oriented $2$-integral varifold on $\R^3$ with compact support, $\mu_{V^0}(\R^3)<\infty$ and $\partial[|V^0|]=0$.
Then the isoperimetric inequality for currents \cite[Thm 30.1]{Simon_Buch} yields an integral $3$-current $R\in D_3(\R^3)$ with
\begin{equation}
 \label{eq:3_1}
 \partial R = [|V^0|],
\end{equation}
$(M(R))^\frac{2}{3}\leq C M([|V^0|]) \leq C \mu_{V^0}(\R^3) < \infty$ for some constant $C>0$ independent of $R$ or $[|V^0|]$ and such that $R$ has compact support.
Since $R$ is three dimensional in $\R^3$, \cite[Remark 26.28]{Simon_Buch} yields a function $\theta_R:\R^3\rightarrow \Z$ of bounded variation, such that for every
$\omega\in D^2(\R^3)$ we have
\begin{equation}
 \label{eq:3_2}
 R(\omega) = \int \langle \omega, e_1\wedge e_2\wedge e_3\rangle \theta_R\, d\mathcal{L}^3.
\end{equation}
Now we claim the following equation, which we will prove afterwards
\begin{equation}
 \label{eq:3_3}
 Vol(V^0) = -\int \theta_R\, d\mathcal{L}^3.
\end{equation}
As a short remark:
If $V^0$ would be given by an embedded closed surface, $R$ would represent the open and bounded set with boundary $V^0$.\\
Since $\theta_R$ is of bounded variation, we find a Borel measure $|\triangledown \theta_R|$ and a Borel measurable function $\sigma_R:\R^3\rightarrow \partial B_1(0)$, 
such that for every smooth $g:\R^3\rightarrow\R^3$ with compact support we have (see e.g. \cite[Section 5.1]{EvansGariepy})
\begin{equation*}
 \int \operatorname{div} g \theta_R\, d\mathcal{L}^3 = -\int \langle g,\sigma_R\rangle\, d|\triangledown \theta_R|. 
\end{equation*}
We now claim
\begin{equation}
 \label{eq:3_4}
 \int \langle g,\sigma_R\rangle\, d|\triangledown \theta_R| = \int \langle g(x),*(\xi)\rangle\, dV^0(x,\xi)
\end{equation}
for every $g\in C^\infty_c(\R^3,\R^3)$.
The proof is as follows: 
We define $\omega:=-g_1dx_2\wedge dx_3 + g_2dx_1\wedge dx_3 - g_3 dx_1\wedge dx_2$ and then \cite[Remark 26.28]{Simon_Buch} yields:
\begin{align*}
 \int \langle g,\sigma_R\rangle\, d|\triangledown \theta_R| =& - \int \theta_R\operatorname{div}(g) \, d\mathcal{L}^3\\
 =& -R(d\omega) = -(\partial R)(\omega) = -[|V^0|](\omega)\\
 =& -\int \langle \omega(x),\xi\rangle\, dV^0(x,\xi). 
\end{align*}
Furthermore we have 
\begin{align*}
 &\xi = \xi_1 e_2\wedge e_3 + \xi_2 e_1\wedge e_3 + \xi_3 e_1\wedge e_2\\
 \Rightarrow & *\xi = \xi_1 e_1 - \xi e_2 + \xi e_3.
\end{align*}
Hence
\begin{align*}
 \langle \omega,\xi\rangle =& -g_1\xi_1 + g_2\xi_2 - g_3\xi_3\\
 \Rightarrow\langle g,*(\xi)\rangle =& g_1\xi_1 - g_2\xi_2 + g_3\xi_3 = -\langle \omega,\xi\rangle.
\end{align*}
which yields \eqref{eq:3_4}.

Since $R$ has compact support and by multiplying with a smooth cutoff function, \eqref{eq:3_4} is valid for every smooth vectorfield.
Hence we can apply \eqref{eq:3_4} to $g(x)=\frac{1}{3}x$ and obtain
\begin{align}
\label{eq:3_4_1}
\begin{split}
 \int1\cdot\theta_R\, d\mathcal{L}^3 =& \frac{1}{3}\int\operatorname{div}(x)\theta_R(x)\, d\mathcal{L}^3(x) = -\frac{1}{3}\int \langle x,\sigma_R\rangle\, d|\triangledown \theta_R|\\
 =& -\frac{1}{3}\int \langle x,*\xi\rangle\, dV^0(x,\xi) = -Vol(V^0),
 \end{split}
\end{align}
which is \eqref{eq:3_2}.
Under our assumptions, i.e. finite mass and compact support, the current $R$ is unique.
Let us prove this claim:
Assume we have $R_1,R_2\in D_3(\R^3)$ with compact support and finite mass satisfying
\begin{equation*}
 \partial R_1 =[|V^0|],\quad \partial R_2 = [|V^0|].
\end{equation*}
Then
\begin{equation*}
 \partial (R_1 - R_2) = 0 
\end{equation*}
and 
\begin{equation*}
 M(R_1-R_2)\leq M(R_1) + M(R_2) < \infty.
\end{equation*}
Now the constancy theorem for currents (see e.g. \cite[Thm. 26.27]{Simon_Buch}) yields a $c\in \R$, such that
\begin{equation*}
 R_1 - R_2 = c[|\R^3|].
\end{equation*}
Since $R_1-R_2$ has finite mass, $c=0$, i.e.
\begin{equation}
 \label{eq:3_5}
 R_1 = R_2
\end{equation}

If we want to calculate the first derivative of the enclosed volume, we need to make sense of mapping an oriented integral varifold by a diffeomorphism.
So let $g:\R^3\rightarrow\R^3$ be a diffeomorphism.
Then we define $g_\sharp V^0:=V^0[\tilde{M},\tilde{\theta}_+,\tilde{\theta}_-,\tilde{\xi}]$ by
\begin{align}
\label{eq:3_6}
\begin{split}
 \tilde{M}:=&g(M)\\
 \tilde{\theta}_\pm(x) :=& \theta_\pm(g^{-1}(x))\\
 \tilde{\xi}(x):=& \left(\frac{d^Mg_{y\sharp}\xi(y)}{\left|d^Mg_{y\sharp}\xi(y)\right|}\right)\bigg|_{y=g^{-1}(x)}.
\end{split}
 \end{align}
Then we have for every $\omega\in D^2(\R^3)$
\begin{align*}
 &[|g_\sharp V^0|](\omega)\\
 =&\int_{g(M)} \left\langle \omega(x), \left(\frac{d^Mg_{y\sharp}\xi(y)}{\left|d^Mg_{y\sharp}\xi(y)\right|}\right)\bigg|_{y=g^{-1}(x)}\right\rangle(\theta_+(g^{-1}(x)) - \theta_-(g^{-1}(x)))\, d\mathcal{H}^2
\end{align*}
By \cite[Remark 27.2]{Simon_Buch} we also have
\begin{align*}
 &(g_\sharp[|V^0|])(\omega)\\
 =&\int_{g(M)} \left\langle \omega(x), \left(\frac{d^Mg_{y\sharp}\xi(y)}{\left|d^Mg_{y\sharp}\xi(y)\right|}\right)\bigg|_{y=g^{-1}(x)}\right\rangle(\theta_+(g^{-1}(x)) - \theta_-(g^{-1}(x)))\, d\mathcal{H}^2
\end{align*}
and hence
\begin{equation}
 \label{eq:3_7}
 g_\sharp[|V^0|] = [|g_\sharp V^0|].
\end{equation}

Now we are ready to calculate the first variation of the enclosed volume.
Let $X\in C^\infty_c(\R^3,\R^3)$ be a vectorfield and $\Phi:\R\times \R^3\rightarrow \R^3$ the corresponding flow, i.e.
$\partial_t (\Phi(t,x))=X(\Phi(t,x))$ and $\Phi(0,x)=x$ for all $t,x$.
Since $x\mapsto\Phi(t,x)$ is a diffeomorphism for every $t$, $\Phi(t,\cdot)_\sharp [|V^0|]$ is well defined. 
\eqref{eq:3_7} now yields with \cite[Remark 26.21]{Simon_Buch}
\begin{equation}
\label{eq:3_8}
 \partial \left( [|\Phi(t,\cdot)_\sharp V^0|]\right)=\partial \left(\Phi(t,\cdot)_\sharp [|V^0|]\right) = \Phi(t,\cdot)_\sharp \left(\partial[|V^0|]\right) = \Phi(t,\cdot)_\sharp R. 
\end{equation}
Hence the uniqueness property of $R$ and \eqref{eq:3_3} give us the first equality
\begin{equation}
 \label{eq:3_9}
 Vol(\Phi(t,\cdot)_\sharp V^0) = -\int \theta_{\Phi(t,\cdot)_\sharp R}\, d\mathcal{L}^3 = -\int \theta_R(\Phi(t,\cdot)^{-1}(x))\, d\mathcal{L}^3(x).
\end{equation}
The second equality follows from \cite[Remark 27.2]{Simon_Buch} and the fact, that $x\mapsto\Phi(t,x)$ is orientation preserving for every $t\in\R$.
If we now decompose $\theta_R$ into positiv and negativ parts, we can employ the calculation for the first variation of varifolds (see e.g. \cite[§ 16]{Simon_Buch})
and we finally obtain
\begin{equation}
 \label{eq:3_10}
 \left(\frac{\partial}{\partial t} Vol(\Phi(t,\cdot)_\sharp V^0)\right)\bigg|_{t=0} = -\int \operatorname{div}(X)\theta_R\, d\mathcal{L}^3 = \int \langle X,\sigma_R\rangle \, d|\triangledown\theta_R|
\end{equation}

\section{Area and volume correction}
\label{sec:4}

In chapter \ref{sec:5} we will apply the graphical decomposition method by Simon, see \cite[§ 3]{Simon} to show partial regularity of the limit. 
This argument entails comparing the minimising sequence to a biharmonic graph in terms of the Helfrich energy.
In order to guarantee that the enclosed volume and area stay the same in this procedure, we need to correct them for the comparing sequence.

Hence we construct in this chapter a two parameter diffeomorphism of $\R^3$, such that we can adjust with it the enclosed volume and area of the changed minimising sequence
outside of the biharmonic comparison region, see e.g. Figure \ref{fig_1}.\\
This idea was introduced by Schygulla in \cite{Schygulla} for a one parameter diffeomorphism and prescribed isoperimetric ratio. 
We will expand this idea by using a version of the inverse function theorem (see Theorem \ref{A_2}) with some explicit bounds on the size of the set of invertibility.
\begin{figure}
 \centering 
\includegraphics{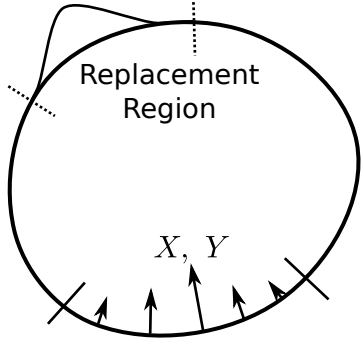}  
\caption{Correcting Area and Volume outside of a prescribed region.}
\label{fig_1}
\end{figure}

Before we can start working on our diffeomorphism, we have to define currents for the minimising sequence $f_k$,
with which we will be able to calculate the enclosed volume as in section \ref{sec:3}.
So let $(R_k)_{k\in\N}\subset D_3(\R^3)$ be sequence of integral currents satisfying (see also \eqref{eq:3_1})
\begin{align}
 \label{eq:4_1}
 \begin{split}
  \partial R_k =& [|V^0_k|]\\
  M(R_k)^\frac{2}{3}\leq &C M([|V^0_k|]).
 \end{split}
\end{align}
Here $C>0$ is given by the isoperimetric inequality \cite[Thm. 30.1]{Simon_Buch}.
As in section \ref{sec:3} there are functions of bounded variation $\theta_{R_k}:\R^3\rightarrow\Z$ representing $R_k$ as in \eqref{eq:3_2}.
Since 
\begin{equation*}
 M(R_k)\leq C M(\partial R_k)^\frac{3}{2} = C M([|V^0_k|])^\frac{3}{2} \leq C \mu_k(\R^3)^\frac{3}{2} = C Vol_0^\frac{3}{2}< \infty 
\end{equation*}
the compactness theorem for currents \cite[Thm. 27.3]{Simon_Buch} rsp. for functions of bounded variation \cite[Section 5.2.3]{EvansGariepy} yield a convergent subsequence of $R_k$ rsp. $\theta_k$ in the current rsp. BV sense.
After relabeling we can hence assume $R_k\rightarrow R$ weakly as currents and $\theta_{R_k}\rightarrow \theta_R$ in the BV sense.
Here $\theta_R:\R^3\rightarrow\Z$ satisfies \eqref{eq:3_2} with respect to $R$.
Since the boundaries of currents also converge, we have
\begin{equation*}
 \partial R = [|V^0|].
\end{equation*}
Since $R$ has finite mass, the uniqueness result \eqref{eq:3_5} also yields $M(R)^\frac{2}{3}\leq CM([|V^0|])$.

Let us define the following sequence of functions corresponding to the area and enclosed volume:
For this we need $X,Y\in C_c^\infty(\R^3,\R^3)$ Let $\Phi_X$ rsp. $\Phi_Y$ be the flow of $X$ rsp. $Y$.
Now let 
\begin{equation}
\label{eq:4_1_1}
 \Phi(s,t,x):= \Phi_X(s,\Phi_Y(t,x))
\end{equation}
for $s,t\in\R$ and $x\in\R^3$. Let
\begin{equation}
 \label{eq:4_2}
 F(s,t):=\left(\begin{array}{c} (\Phi(s,t,\cdot)_\sharp\mu)(\R^3)\\ Vol(\Phi(s,t,\cdot)_\sharp V^0)\end{array}\right).
\end{equation}
and 
\begin{equation}
 \label{eq:4_3}
 F_k(s,t):=\left(\begin{array}{c} (\Phi(s,t,\cdot)_\sharp\mu_k)(\R^3)\\ Vol(\Phi(s,t,\cdot)_\sharp V^0_k)\end{array}\right).
\end{equation}
Since $\Phi$ does not depend on $k$, we readily obtain $F^1_k\rightarrow F^1$ pointwise as $k\rightarrow\infty$.
Furthermore $\theta_{R_k}$ converges in the BV-sense and we therefore have $L^1$-convergence $\theta_{R_k}\rightarrow \theta_R$.
Hence for all $s,t\in\R$
\begin{align*}
&Vol(\Phi(s,t,\cdot)_\sharp V^0_k) \overset{\eqref{eq:3_9}}{=} -\int \theta_{R_k}(\Phi(s,t,\cdot)^{-1}(x))\, d\mathcal{L}^3(x)\\
 =& -\int \theta_{R_k} |\det(D\Phi)|\, d\mathcal{L}^3 \rightarrow -\int \theta_R |\det(D\Phi)|\, d\mathcal{L}^3\\
 =& -\int\theta_{R}(\Phi^{-1}(s,t,x))\,d\mathcal{L}^3 \overset{\eqref{eq:3_9}}{=} Vol(\Phi(s,t,\cdot)_\sharp V^0),
\end{align*}
which yields the pointwise convergence $F_k\rightarrow F$.
Here ${}^{-1}$ and $D$ refer to the $x$-variable of $\Phi$.\\
Since we want to locally invert $F$ and $F_k$, we have to derive $F$ and find $X,Y$ such that $DF(0,0)$ is invertible.
\eqref{eq:3_10} and the usual calculation of the first variation of a varifold (see e.g. \cite[§16]{Simon_Buch}) yields
\begin{equation}
 \label{eq:4_4}
 DF(0,0) = \left(\begin{array}{cc} \int \operatorname{div}_{T_x\mu} X(x)\, d\mu(x) & \int \operatorname{div}_{T_x\mu} Y(x)\, d\mu(x)\\
                  -\int \theta_R\operatorname{div} X\, d\mathcal{L}^3 & -\int \theta_R \operatorname{div} Y\, d\mathcal{L}^3 
                 \end{array}\right)
\end{equation}

To find the necessary vectorfields, we work similarly to \cite[Lemma 4]{Schygulla} and prove the following two lemmas: 
\begin{lemma}[cf. Lemma 4 in \cite{Schygulla}]
 \label{4_1}
 There exists an $r>0$, such that for every $x\in \supp(\mu)$ exists a point $\eta \in \supp(\mu)\setminus B_{r}(x)$, 
 such that for every $\varepsilon>0$ there exists a vectorfield $Y\in C^\infty_0(B_\varepsilon(\eta),\R^3)$ satisfying
 \begin{equation*}
  \int \theta_R \operatorname{div} Y\, dx = -\int \langle Y, \sigma_R\rangle\, d|\triangledown \theta_R|\neq 0. 
 \end{equation*}
\end{lemma}
\begin{proof}
 We proceed by contradiction and assume the statement is false. 
 Then there exist sequences $r_k\rightarrow 0$ and $x_k\in \supp(\mu)$, such that for all $\eta\in \supp(\mu)\setminus B_{r_k}(x_k)$ there is an $\varepsilon_\eta>0$ such that
 \begin{equation*}
    \int \langle Y, \sigma_R\rangle \, d|\triangledown\theta_R| =0.
 \end{equation*}
for all $Y\in C^\infty_0(B_{\varepsilon_\eta}(\eta),\R^3)$. 
By Lemma \ref{2_1} $\supp(\mu)$ is compact and thus we can extract a subsequence and relabel such that $x_k\rightarrow x\in \supp(\mu)$.
Hence we obtain
\begin{equation*}
 \sigma_R=0\ \ |\triangledown\theta_R|\mbox{-a.e.}.
\end{equation*}
Thus by \eqref{eq:3_4_1}
\begin{equation*}
 Vol(V^0)=0,
\end{equation*}
which is a contradiction to $Vol(V^0)=Vol_0\neq 0$.
\end{proof}

\begin{lemma}[cf. Lemma 4 in \cite{Schygulla}]
 \label{4_2}
 Either $H_\mu\in L^\infty(\mu)$ or
 there exists an $r>0$, such that for every $z\in \supp(\mu)$ there exist two points $\eta_Y,\eta_X\in \supp(\mu)\setminus B_{r}(z)$, 
 such that for every $\varepsilon_Y,\varepsilon_X>0$ there exist vectorfields $Y\in C^\infty_0(B_{\varepsilon_Y}(\eta_Y))$ and $X\in C^\infty_0(B_{\varepsilon_X}(\eta_X))$ satisfying
 \begin{align*}
  &\int \operatorname{div}_{T_x\mu} X(x)\, d\mu(x) \int \langle Y,\sigma_R\rangle\, d|\triangledown\theta_R|\\
  -& \int \operatorname{div}_{T_x\mu} Y(x)\, d\mu(x) \int \langle X,\sigma_R\rangle\, d|\triangledown\theta_R|\neq0. 
 \end{align*}
\end{lemma}
\begin{proof}
Let us assume that we do not find an $r>0$ as requested. Then we find sequences $r_k\searrow 0$, $z_k\in \supp(\mu)$ such that for every
$\eta_Y,\eta_X\in \supp(\mu)\setminus B_{r_k}(z_k)$ there are $\varepsilon_{Y,\eta_Y},\varepsilon_{X,\eta_X}>0$, which satisfy
 \begin{align}
 \begin{split}
 \label{eq:4_5}
  &\int \operatorname{div}_{T_x\mu} X(x)\, d\mu(x) \int \langle Y,\sigma_R\rangle\, d|\triangledown\theta_R|\\
  -& \int \operatorname{div}_{T_x\mu} Y(x)\, d\mu(x) \int \langle X,\sigma_R\rangle\, d|\triangledown\theta_R|=0
 \end{split}
 \end{align}
 for every $Y\in C^\infty_0(B_{\varepsilon_{Y,\eta_Y}}(\eta_Y))$ and $X\in C^\infty_0(B_{\varepsilon_{X,\eta_X}}(\eta_X))$.
Since $\supp(\mu)$ is compact, we find a convergent subsequence of $z_k$ and after relabeling we have $z_k\rightarrow z\in \supp(\mu)$.
Hence for every $\eta_Y,\eta_X\in \supp(\mu)\setminus\{z\}$ there are $\varepsilon_{Y,\eta_Y},\varepsilon_{X,\eta_X}>0$, such that
every $Y\in C^\infty_0(B_{\varepsilon_{Y,\eta_Y}}(\eta_Y),\R^3)$ and $X\in C^\infty_0(B_{\varepsilon_{X,\eta_X}}(\eta_X),\R^3)$ satisfies \eqref{eq:4_5}.
According to Lemma \ref{4_1} we find an $\eta\in \supp(\mu)\setminus\{z\}$ and a $Y\in C^\infty_0(B_{\varepsilon_{Y,\eta}}(\eta))$ with 
\begin{equation*}
 \int \langle Y, \sigma_R\rangle \, d|\triangledown\theta_R| \neq 0
\end{equation*}
From now on this $Y$ is fixated.
All in all we find for every $\eta_X\in \supp(\mu)\setminus\{z\}$ a radius $\varepsilon_{\eta_X}$ such that for every $X\in C^\infty_0(B_{\varepsilon_{X,\eta_X}})$ we have
 \begin{align*}
  &\int \operatorname{div}_{T_x\mu} X(x)\, d\mu(x) \int \langle Y,\sigma_R\rangle\, d|\triangledown\theta_R|\\
  -& \int \operatorname{div}_{T_x\mu} Y(x)\, d\mu(x) \int \langle X,\sigma_R\rangle\, d|\triangledown\theta_R|=0
 \end{align*}
 Hence 
  \begin{align*}
 \int_{\supp(\mu)} \langle H_\mu, X\rangle (\theta_++\theta_-)\,d\mathcal{H}^2 =& -\int \operatorname{div}_{T_x\mu} X(x)\, d\mu(x)\\
 =& C(\varepsilon_{Y})\int \langle X,\sigma_R\rangle\, d|\triangledown\theta_R|\\
  \overset{\eqref{eq:3_4}}{=} &C(\varepsilon_{Y}) \int \langle X,(*\xi)\rangle\, dV^0\\
  =& C(\varepsilon_{Y})\int_{\supp(\mu)} \langle X(x),(*\xi(x))\rangle (\theta_+-\theta_-)\, d\mathcal{H}^2(x)
 \end{align*}
 Since $\eta_X$ is arbitrary and $\mu(\{z\})=0$ we have 
\begin{equation}
\label{eq:4_6}
 H_{\mu} (\theta_+ + \theta_-) = C(\varepsilon_{Y})\cdot (*\xi)(\theta_+-\theta_-)\quad \mu\mbox{-a.e.}.
\end{equation}
Furthermore $\theta_+ + \theta_-\geq 1$ $\mu$-a.e. and $\theta_\pm\in\N_0$ finally yield $H_\mu\in L^\infty(\mu)$.

\end{proof}




Since we like to apply Theorem \ref{A_2} to $F_k$, we need to be able to estimate the difference of two values of the first derivative independently of $k$.
To do this, we employ the mean value theorem and hence need to estimate the second derivative:
\begin{lemma}
 \label{4_4}
 For every $X,Y\in C^\infty_c(\R^3,\R^3)$ and every $T>0$, there exists a constant $C=C(\|\Phi\|_{C^3(B_{T}(0)\times\R^3)}, Area_0,Vol_0)>0$, 
 such that for every $s,t\in B_{T}(0)$ we have for every $k\in\N$:
 \begin{equation*}
  |D^2 F_k(s,t)|\leq C.
 \end{equation*}
\end{lemma}
\begin{proof}
We start by estimating $F^2$:\\
Since \eqref{eq:3_9} only needs the corresponding map to be a diffeomorphism, we obtain the same result for $x\mapsto \Phi(s,t,x)$.
Furthermore the usual substitution formula yields:
\begin{align}
 \label{eq:4_7}
 \begin{split}
 &F^2_k(s,t) =  -\int \theta_{R_k}(\Phi(s,t,\cdot)^{-1}(x))\, d\mathcal{L}^3(x)\\ =& -\int |\det(D_x\Phi(s,t,x))|\theta_{R_k}(x)\, d\mathcal{L}^3(x).
\end{split}
 \end{align}
Now there exists a constant $C>0$ only dependend on $T>0$, $X$, $Y$ and their respective derivatives, such that
for every $s,t \in (-T,T)$ we have
\begin{equation*}
 D_{s,t}|\det(D_x\Phi(s,t,x))|, D_{s,t}^2|\det(D_x\Phi(s,t,x))| \leq C.
\end{equation*}
Hence for these $s,t$ the isoperimetric inequality \cite[Thm 30.1]{Simon_Buch} yields
\begin{equation*}
 |D^2F^2_k (s,t)|\leq C\int |\theta_{R_k}|\, d\mathcal{L}^3 \leq C(\mu_k(\R^3))^\frac{3}{2} = C(Area_0)^\frac{3}{2}.
\end{equation*}
Now let us turn to $F^1$:\\
Since $x\mapsto \Phi(s,t, x)$ is a diffeomorphism, we can apply \cite[Eq. 15.7]{Simon_Buch} and obtain
\begin{equation}
 \label{eq:4_8}
 F^1_k(s,t)= \int (J_{\mu_k} \Phi(s,t,x))\, d\mu_k.
\end{equation}
Furthermore in the Jacobian $J_{\mu_k}\Phi(s,t,x)$ the measure $\mu_k$ is independent of $s$ and $t$.
Also the derivatives $d^{\mu_k}\Phi(s,t,x)$ appearing in the jacobian, can be estimated by the full derivative of $x\mapsto\Phi(s,t,x)$.
Therefore there is a constant $C>0$ only dependend on $T>0$, $X$, $Y$ and their respective derivatives, such that
for every $s,t\in(-T,T)$ we have
\begin{equation*}
 |J_{\mu_k} \Phi(s,t,x)|\leq C.
\end{equation*}
This yields
\begin{equation*}
 |D^2F^1(s,t)| \leq C\mu_k(\R^3) = C\cdot Area_0,
\end{equation*}
which is the desired conclusion.
\end{proof}

If we change the minimising sequence by $\Phi$, we also have to be sure, that the Helfrich energy and the second fundamental form is controlled as well:
\begin{lemma}[cf. \cite{Schygulla} p. 915, Eq. (v)]
 \label{4_5}
  For every $X,Y\in C^\infty_c(\R^3,\R^3)$ and every $T>0$, there exists a constant $C=C(\|\Phi\|_{C^3(B_{T}(0)\times\R^3)}, Area_0,Vol_0)>0$ (independent of $k$!),
  such that for every $s,t\in (-T,T)$ we have
  \begin{equation*}
   \left|D_{s,t}\int |A_{k}^{s,t}|^2\, d\mu_k\right|\leq C. 
  \end{equation*}
Here $A^{s,t}_{k}$ is the second fundamental form of $\Sigma\ni x\mapsto\Phi(s,t, f_k(x))$.
\end{lemma}
\begin{proof}
 By a partition of unity on $\Sigma$ and a rigid motion we can assume, that we can write $x\mapsto\Phi(s,t,f_k(x))$ locally as a smooth graph with small Lipschitz norm.
 Hence we have $u_k:\R^2\supset B_r(0)\rightarrow\R$ smooth with $|\nabla u_k|\leq 1$ and $u_k(0)=0$, which satisfies for a small open set $U\subset \Sigma$:
 \begin{equation*}
  f_k(U)=\operatorname{graph}(u_k).
 \end{equation*}
Now we can calculate the second fundamental form of $\Phi(s,t,f_k(\cdot))$ by using $u_k$ and chain rule:
\begin{align}
 \label{eq:4_9}
 \begin{split}
 \partial_i\left(\Phi(s,t,(x,u_k(x)))\right) =& \sum_{j=1}^2(\partial_j\Phi)(s,t,(x,u_k(x)))\delta_{ij}\\& + (\partial_3\Phi)(s,t,(x,u_k(x)))\partial_i u_k(x)\\
 =& (\partial_i\Phi)(s,t,(x,u_k(x))) + (\partial_3\Phi)(s,t,(x,u_k(x)))\partial_i u_k(x).
\end{split}
 \end{align}
Then the second derivatives are as follows:
\begin{align}
\label{eq:4_10}
\begin{split}
\partial_\ell \partial_i\left(\Phi(s,t,(x,u_k(x)))\right) = & \sum_{m=1}^2 (\partial_m\partial_i\Phi)(s,t,(x,u_k(x))\delta_{m\ell} \\
&+ (\partial_3\partial_i\Phi)(s,t,(x,u_k(x))\partial_\ell u_k(x)\\
& + \sum_{m=1}^2(\partial_m \partial_3\Phi)(s,t,(x,u_k(x))) \delta_{m\ell}\partial_i u_k(x)\\
& + (\partial_3^2\Phi)(s,t,(xu_k(x)))\partial_\ell u_k(x)\partial_iu_k(x)\\
& + \partial_3\Phi(s,t,(x,u_k(x)))\partial_\ell\partial_iu_k(x).
 \end{split}
\end{align}
The unit normal can be expressed by 
\begin{equation*}
 n(s,t,x)=\pm\frac{\partial_1\left(\Phi(s,t,(x,u_k(x)))\right)\times \partial_2\left(\Phi(s,t,(x,u_k(x)))\right)}{\left|\partial_1\left(\Phi(s,t,(x,u_k(x)))\right)\times \partial_2\left(\Phi(s,t,(x,u_k(x)))\right)\right|}
\end{equation*}
Since $x\mapsto\Phi(s,t,x)$ is a diffeomorphism of $\R^3$, such that it is the identity outside of a ball, we find a constant $C=C(T)>0$, 
such that for every $(s,t)\in B_{T}(0)$ we have
\begin{equation*}
|\det D_x\Phi(s,t,x)|\geq C.
\end{equation*}
This yields a noncollapsing of the basis of the tangential space, i.e.
\begin{equation*}
 \left|\partial_1\left(\Phi(s,t,(x,u_k(x)))\right)\times \partial_2\left(\Phi(s,t,(x,u_k(x)))\right)\right|\geq C>0
\end{equation*}
for a possibly different constant $C=C(T)>0$.
The second fundamental form can be expressed in these coordinates as
\begin{equation*}
 (A_{k}^{s,t})_{i\ell} = \langle \partial_\ell \partial_i\left(\Phi(s,t,(x,u_k(x)))\right) , n(s,t,x)\rangle.
\end{equation*}
All in all we therefore have
\begin{equation*}
 D_{s,t}\left(|A_{k}^{s,t}|^2\right)\leq C\left(1+\|\nabla u_k\|_{L^\infty}^4 + \|D^2 u_k\|_{L^\infty}^2\right)
\end{equation*}
with some constant $C>0$ independent of $k$. 
Since we imposed a small Lipschitz Norm on $u_k$ we deduce by e.g. \cite[p. 5]{DeckGruRoe}
\begin{equation*}
 D_{s,t}\left(|A_{k}^{s,t}|^2\right)\leq C(1+|A_{k}|^2).
\end{equation*}
Hence
\begin{equation*}
 \left|D_{s,t}\int |A_{k}^{s,t}|^2\, d\mu_k\right|\leq C
\end{equation*}
and $C$ does not depend on $k$.
\end{proof}

\begin{remark}
\label{4_6}
With the same techniques as in Lemma \ref{4_5} we can also show (cf. \eqref{eq:4_7})
\begin{equation*}
 |D_{s,t}W_{H_0}(f_k^{s,t}) + D_{s,t}\operatorname{Area}(f_k^{s,t})| \leq C
\end{equation*}
for all $(s,t)\in B_{T}(0)$
and $C$ is independent of $k$. Here $f_k^{s,t}$ is the immersion $\Sigma\ni x\mapsto \Phi(s,t,f_k(x))$.
\end{remark}

\section{Partial regularity and lower-semicontinuity}
\label{sec:5}
In this section we adapt the partial regularity method introduced by Simon (see \cite[Section 3]{Simon}).
This method is based on replacing parts of the minimising sequence with biharmonic graphs and compare the resulting energies. 
Here we use the idea of Schygulla \cite[Lemma 5]{Schygulla} to correct the area and enclosed volume of the modified sequence, 
so that they become competitors for the minimum of the Helfrich energy again.

Let $\varepsilon_0>0$ be fixated. 
In dependence of this $\varepsilon_0$ we say $x_0\in \operatorname{spt}(\mu)$ is a good point iff (see \eqref{eq:2_8})
\begin{equation*}
 \nu(\{x_0\}) < \varepsilon_0^2.
\end{equation*}
In neighbourhoods of these good points we will show $C^{1,\beta}$ regularity and a graphical decomposition of $\mu$.
Since we work with immersions with possible self intersections, we also use the ideas of Schätzle 
(see \cite[Prop. 2.2]{Schaetzle}) to implement Simon's regularity method (cf. \cite[Section 3]{Simon}).
The lemma is now as follows (cf. also \cite[Lemma 3.1]{EichmannHelfrichBoundary}):

\begin{lemma}
 \label{5_1}
  For any $\varepsilon>0$ there exist $\varepsilon_0 = \varepsilon_0(H_0, Vol_0, Area_0, \varepsilon)>0$, $\theta=\theta(H_0, Vol_0, Area_0,\varepsilon)>0$,
  $\rho_0 = \rho_0(H_0, Vol_0, Area_0,\varepsilon)>0$ and
 $\beta = \beta(H_0, Vol_0, Area_0) > 0$, such that for every good point $x_0\in \operatorname{spt}(\mu)$ and good radius $0<\rho_{x_0}\leq \rho_0$ satisfying
 \begin{equation}
 \label{eq:5_1}
  \nu(\overline{B_{\rho_{x_0}}(x_0)}) < \varepsilon_0^2,
 \end{equation}
 $\mu$ is a union of $(W^{2,2}\cap C^{1,\beta})$-graphs in $B_{\theta\rho_{x_0}}(x_0)$ 
 of functions $u_i\in (W^{2,2}\cap C^{1,\beta})(B_{\theta \rho_{x_0}}(x_0)\cap L_i)$. 
 Here $L_i\subset \R^3$ are two dimensional affine spaces and $i=1,\ldots, I_{x_0}\leq C(Area_0,Vol_0, H_0)$. 
 Furthermore the $u_i$ satisfy the following estimate
 \begin{align}
  \begin{split}
  \label{eq:5_2}
  &(\theta \rho_{x_0})^{-1}\|u_i\|_{L^\infty\left(B_{\theta\rho_{x_0}}(x_0)\cap L_i\right)}\\
  +\ & \|\nabla u_i\|_{L^\infty\left(B_{\theta\rho_{x_0}}(x_0)\cap L_i\right)} + (\theta\rho_{x_0})^\beta h\ddot{o}l_{B_{\theta\rho_{x_0}}(x_0)\cap L_i,\beta}\left(\nabla u_i\right)\leq \varepsilon.
 \end{split}
 \end{align}
Moreover we have a power-decay for the second fundamental form, i.e. $\forall x\in B_{\frac{\theta\rho_{x_0}}{4}}(x_0)$, $0<\rho<\frac{\theta \rho_{x_0}}{4}$
\begin{equation}
 \label{eq:5_3}
 \int_{B_\rho(x)}|A_{\mu}|^2\, d(\mu)\leq C(H_0,Area_0,Vol_0)(\varepsilon_0^2+\rho_0)\rho^\beta\rho_{x_0}^{-\beta}.
\end{equation}
\end{lemma}
\begin{proof}
 We start by applying the graphical decomposition lemma of Simon (see \cite[Lemma 2.1]{Simon}) to the minimising sequence $f_k:\Sigma\rightarrow \R^3$.
 This lemma is also applicable for immersions by an argument of Schätzle 
 (see \cite[p. 280]{Schaetzle} and cf. \cite[Lemma A.6]{EichmannHelfrichBoundary}, in the beginning we work as in these papers).
 We repeat some steps of \cite[(2.11)-(2.16)]{Schaetzle}, which we will need later (see also \cite[(3.5)-(3.14)]{EichmannHelfrichBoundary}):\\
 The upper-semicontinuity of the weak convergence $|A_k|^2\mu_k\rightarrow \nu$ for Radon measures (see e.g. \cite[Prop. 4.26]{Maggi}) yields for every Ball satisfying \eqref{eq:5_1}
 \begin{equation}
 \label{eq:5_4}
  \limsup_{k\rightarrow\infty}\int_{B_{\rho_{x_0}}(x_0)}|A_k|^2\, d\mu_k \leq \nu(\overline{B_{\rho_{x_0}}(x_0)}) < \varepsilon_0^2.
 \end{equation}
Hence the aforementioned graphical decomposition lemma \cite[Lemma A.6]{EichmannHelfrichBoundary} is applicable 
and we can decompose $f^{-1}_k (\overline{B_{\frac{\rho_{x_0}}{2}}(x_0)})$ for large $k$ and $\varepsilon_0>0$ small enough
into closed pairwise disjoint sets $D_{k,i}\subset \Sigma$ (which are topological discs), $i=1,\ldots, I_k$, $I_k\leq C E_{H_0,Area_0,Vol_0}$ (cf. \eqref{eq:2_6}, \eqref{eq:1_6}), i.e.
\begin{equation*}
 f^{-1}_k\left(\overline{B_{\frac{\rho_{x_0}}{2}}(x_0)}\right) = \bigcup_{i=1}^{I_k} D_{k,i}.
\end{equation*}
Furthermore for every $i=1,\ldots, I_k$ we have affine $2$-dimensional planes $L_{k,i}\subset \R^3$, 
smooth functions $u_{k,i}:\overline{\Omega}_{k,i}\subset L_{k,i}\rightarrow L_{k,i}^\perp$ 
($\Omega_{k,i}= \Omega_{k,i}^0\setminus \cup_m d_{k,i,m}$, $\Omega_{k,i}^0\subset L_{k,i}$ simply connected, $d_{k,i,m}$ closed pairwise disjoint discs), satisfying
\begin{equation}
 \label{eq:5_5}
 \rho_{x_0}^{-1} |u_{k,i}| + |\nabla u_{k,i}| \leq C(E_{H_0,Area_0,Vol_0}) \varepsilon_0^{\frac{1}{22}}
\end{equation}
and pairwise disjoint topological discs $P_{k,i,1},\ldots, P_{k,i,J_k}\subset D_{k,i}$, such that
\begin{equation}
\label{eq:5_6}
 f_{k}\left(D_{k,i}-\bigcup_{j=1}^{J_{k,i}}P_{k,i,j}\right)=\operatorname{graph}(u_{k,i})\cap\overline{B_{\frac{\rho_{x_0}}{2}}(x_0)}
\end{equation}
and
\begin{equation}
\label{eq:5_7}
 \sum_{i=1}^{I_k}\sum_{j=1}^{J_k} \operatorname{diam} f_{k}(P_{k,i,j})\leq C(E_{H_0,Area_0,Vol_0})\varepsilon_0^{\frac{1}{2}}\rho_{x_0}.
\end{equation}
The arguments \cite[Eq. (2.12)-(2.14)]{Schaetzle} yield for a chosen $0<\tau<\frac{1}{2}$ an $\varepsilon_0$ small enough and $0<\theta<\frac{1}{4}$ such that
\begin{equation}
 \label{eq:5_8}
 \frac{\mu_{g_{k}}(D_{k,i}\cap f_{k}^{-1}(B_{\sigma}(x)))}{w_2\sigma^2}<1+\tau
\end{equation}
for $B_\sigma(x)\subset B_{\theta\rho_{x_0}}(x_0)$ arbitrary. 
Here $w_2$ denotes the Hausdorff measure of the $2$-dimensional euclidean unit ball and $\mu_{g_{k}}$ the area measure on $\Sigma$ induced by $f_{k}$. 
As in \cite[p. 281]{Schaetzle} rsp. \cite[Eq. (3.8)]{EichmannHelfrichBoundary} we define Radon measures on $\R^3$, which will lead to a decomposition of $\mu$ by Radon measures of Hausdorff densitity one:
\begin{align}
 \begin{split}
 \label{eq:5_9}
 \mu_{k,i}&:=\mathcal{H}^2\lfloor f_{k}(D_{k,i}\cap f_{k}^{-1}(B_{\theta\rho_{x_0}}(x_0)))\\ &= f_{k}\left(\mu_{g_{k}}\lfloor (D_{k,i}\cap f_{k}^{-1}(B_{\theta\rho_{x_0}}(x_0)))\right).
\end{split}
 \end{align}
Since the $\mu_{k,i}$ are integer rectifiable, they are of density one, by \eqref{eq:5_8}. 
Furthermore we have as in \cite[p. 281]{Schaetzle} rsp. \cite[Eq. (3.9)]{EichmannHelfrichBoundary}
\begin{equation}
\label{eq:5_10}
 \sum_{i=1}^{I_k}\mu_{k,i} = \mu_{k}\lfloor B_{\theta\rho_{x_0}}(x_0).
\end{equation}
After taking a subsequence depending on $x_0$, $\theta\rho_{x_0}$ and relabeling we can assume $I_k=I$ and 
\begin{equation}
 \label{eq:5_11}
 \mu_{k,i}\rightarrow \mu_i\mbox{ weakly as varifolds in }B_{\theta\rho_{x_0}}(x_0).
\end{equation}
As shown in \cite[Eq. (3.13)]{EichmannHelfrichBoundary} (see also \cite[p. 281]{Schaetzle}) we get
\begin{equation}
 \label{eq:5_12}
 \mu\lfloor B_{\theta\rho_{x_0}}(x_0) = \sum_{i=1}^I\mu_i.
\end{equation}
As in \cite[Eq. (2.16)]{Schaetzle} we can pass to the limit in \eqref{eq:5_8} and obtain
\begin{equation}
 \label{eq:5_13}
 \frac{\mu_i(B_\sigma(x))}{w_2\sigma^2}\leq 1+\tau\quad \forall\ \overline{B_\sigma(x)}\subset B_{\theta\rho_{x_0}}(x_0). 
\end{equation}
We will apply Allard's regularity theorem \ref{A_1} to the $\mu_i$. 
\eqref{eq:5_13} already takes care of the needed density estimate \eqref{eq:A_3}. 
The remainder of the proof will be about showing \eqref{eq:A_2} or a similar $L^p$-bound for the mean curvature.
Here we need to make a distinction of cases given by Lemma \ref{4_2}.\\
First we assume $H_{\mu}\in L^\infty$:\\
In this case we will show, that $H_{\mu_i}\in L^p$ , $p> 2$ arbitrary.
By the usual regularity Theorem of Allard (see e.g. \cite[Thm. 24.2]{Simon_Buch}) the $\mu_i$ will be $C^{1,\beta}$ graphs.

We need the definition of the tilt and height excess, which we repeat here (see e.g. \cite[Eqs. (1.1),(1.2)]{SchaetzleQuadraticTilt}):
\begin{equation}
 \label{eq:5_14}
 \operatorname{tiltex}_\mu(x,\omega,T) := \omega^{-2}\int_{B_\omega(x)}\|T_\xi\mu - T\|^2\, d\mu(\xi),
\end{equation}
\begin{equation}
 \label{eq:5_15}
 \operatorname{heightex}_\mu(x,\omega,T):= \omega^{-4} \int_{B_\omega(x)}\operatorname{dist}(\xi-x,T)^2\, d\mu(\xi).
\end{equation}
Here $x\in\R^3$, $\omega\in\R$ and $T\subset\R^3$ is a two dimensional subspace (cf. \cite[§38]{Simon_Buch} for defining a norm on subspaces, i.e. the unoriented Grassmannian).
Since $H_\mu\in L^\infty(\mu)$, \cite[Thm. 5.1]{SchaetzleQuadraticTilt} yields
\begin{equation}
 \label{eq:5_16}
 \operatorname{tiltex}_\mu(x,\omega,T), \operatorname{heightex}_\mu(x,\omega,T)\leq O_x(\omega^2).
\end{equation}
By the defintion of the height and tilt excess we obtain for every $i=1,\ldots,I$
\begin{equation*}
  \operatorname{tiltex}_{\mu_i}(x,\omega,T)\leq  \operatorname{tiltex}_\mu(x,\omega,T)\leq O_x(\omega^2)
\end{equation*}
and
\begin{equation*}
  \operatorname{heightex}_{\mu_i}(x,\omega,T)\leq  \operatorname{heightex}_\mu(x,\omega,T)\leq O_x(\omega^2),
\end{equation*}
because $\mu_i\leq \mu$. 
Furthermore by \cite[Thm. 3.1]{SchaetzleCurrentLower} $\mu_i$ and $\mu$ are $C^2$-rectifiable.
Hence \cite[Cor. 4.3]{SchaetzleCurrentLower} is applicable and yields
\begin{equation}
 \label{eq:5_17}
 H_{\mu_i}=H_\mu,\quad \mu_i\mbox{-a.e}.
\end{equation}
Since $H_\mu\in L^\infty(\mu)$ and $\mu_i\leq \mu$ we finally obtain
\begin{equation}
 \label{eq:5_18}
 H_{\mu_i}\in L^\infty(\mu_i),
\end{equation}
which yields by the usual Allard regularity theorem (see e.g. \cite[Thm. 24.2]{Simon_Buch}), that every $\mu_i$ is a $C^{1,\beta}$-graph with $\beta\in(0,1)$ arbitrary.

This case is therefore done. \\
Now we modify Schygulla's argument \cite[Lemma 5]{Schygulla} to our situation:
The beginning of the argument is as in \cite[Lemma 3.1]{Simon}. 
For the readers convenience and because we need the notation, we repeat these steps here (see also \cite[pp. 9-10]{EichmannHelfrichBoundary}):

Let us choose $0<\rho<\theta\rho_{x_0}$ fixated but arbitrary. 
We need to apply the graphical decomposition Lemma \cite[Lemma 2.1]{Simon} again to $f_{k}(D_{k,i})\cap \overline{B_\rho(x_0)}$. 
Hence we obtain smooth functions 
$v_{k,i,\ell}:\overline{\tilde{\Omega}_{k,i,\ell}}\subset \tilde{L}_{k,i,\ell}\rightarrow \tilde{L}_{k,i,\ell}^\perp$, $\tilde{L}_{k,i}\subset\R^3$ $2$-dimensional planes
($\ell =1,\ldots, N_{k,i}\leq C(E_{H_0,Area_0,Vol_0})$),
$\tilde{\Omega}_{k,i}=\tilde{\Omega}^0_{k,i}\setminus\cup_m \tilde{d}_{k,i,\ell,m}$, $\tilde{\Omega}^0_{k,i}$ simply connected and $\tilde{d}_{k,i,\ell,m}$ closed pairwise disjoint discs.
Furthermore we have
\begin{equation}
\label{eq:5_21}
 \rho^{-1}|v_{k,i,\ell}| + |\nabla v_{k,i,\ell}|\leq C(E_{H_0,Area_0,Vol_0})\varepsilon_0^\frac{1}{22}
\end{equation}
and closed pairwise disjoint topological discs $\tilde{P}_{k,i,\ell,1},\ldots, \tilde{P}_{k,i,\ell,J_{k,i,\ell}}\subset \tilde{D}_{k,i,\ell}$
($\tilde{D}_{k,i,\ell}$ is a topological disc as well) such that for all $\ell$
\begin{equation}
\label{eq:5_22}
 f_{k}\left(\tilde{D}_{k,i,\ell}-\bigcup_{j=1}^{J_{k,i,\ell}}\tilde{P}_{k,i,\ell,j}\right)\cap\overline{B_{\rho}(x_0)}=\operatorname{graph}(v_{k,i,\ell})\cap\overline{B_{\rho}(x_0)}
\end{equation}
These $\tilde{P}_{k,i,\ell,j}$ also satisfy the following estimate
\begin{equation}
 \label{eq:5_23}
 \sum_{j=1}^{J_{k,i,\ell}} \operatorname{diam} f_{k}(\tilde{P}_{k,i,\ell,j})\leq C(E_{H_0,Area_0,Vol_0})\varepsilon_0^\frac{1}{2}\rho \leq {\frac{1}{8}}\rho,
\end{equation}
if we choose $C(E_{H_0,Area_0,Vol_0})\varepsilon_0^\frac{1}{2}<\frac{1}{8}$ (The need for \eqref{eq:5_23} is also the reason for applying the graphical decomposition a second time).
Let us also introduce the corresponding Radon measures similar to \eqref{eq:5_9}
\begin{equation}
\label{eq:5_24}
 \tilde{\mu}_{k,i,\ell}:=\mathcal{H}^2\lfloor f_{k}(\tilde{D}_{k,i,\ell}) = \mu_{k,i}\lfloor f_{k}(\tilde{D}_{k,i,\ell}).
\end{equation}
Since $f_{k}|_{D_{k,i}\cap f^{-1}_{k}(B_{\theta\rho_{x_0}}(x_0))}$ is an embedding, we also have
\begin{equation}
\label{eq:5_25}
 \sum_{\ell=1}^{N_{k,i}} \tilde{\mu}_{k,i,\ell} = \mu_{k,i}\lfloor B_\rho(x_0).
\end{equation}
Let us define 
\begin{equation}
 \label{eq:5_20}
 C_\sigma^{k,i,\ell}(x_0):=\left\{x+y:\ x\in B_\sigma(x_0)\cap \tilde{L}_{k,i,\ell},\ y\in \tilde{L}_{k,i,\ell}^\perp\right\}.
\end{equation}
Inequality \eqref{eq:5_23} yields $\mathcal{L}^1$-measurable sets $S_k\subset (\frac{1}{2}\rho,\frac{3}{4}\rho)$, such that $\forall j=1,\ldots,J_{k,i,\ell}$ 
\begin{equation*}
 \mathcal{L}^1(S_k)\geq \frac{1}{8}\rho\mbox{ and }\forall \sigma\in S_k:\ \partial C_\sigma^{k,i,\ell}(x_0)\cap f(\tilde{P}_{k,i,\ell,j})=\emptyset.
\end{equation*}
Therefore $v_{k,i,\ell}|_{\partial B_\sigma(x_0)\cap \tilde{L}_{k,i,\ell}}$ and $\nabla v_{k,i,\ell}|_{\partial B_\sigma(x_0)\cap \tilde{L}_{k,i,\ell}}$ are well defined for any $\sigma\in S_k$. 
Hence Lemma \ref{A_3} is applicable and yields a function $w_{k,i,\ell}:B_\sigma(x_0)\cap \tilde{L}_{k,i,\ell}\rightarrow \tilde{L}_{k,i,\ell}^\perp$
with Dirichlet boundary data given by $v_{k,i,\ell}$ and $\nabla v_{k,i,\ell}$. \\ 
This defines a sequence of immersions $f^{graph}_{k,\sigma}:\Sigma_{k,\sigma}\rightarrow\R^3$ by
\begin{equation*}
 \Sigma_{k,\sigma}:=\left(\Sigma\setminus \left(f^{-1}_k(B_\sigma(x_0))\cap \tilde{D}_{k,i}\right)\right)\oplus \left(B_\sigma(x_0)\cap \tilde{L}_{m,i}\right)
\end{equation*}
and
\begin{equation*}
 f^{graph}_{k,\sigma}(p) = \left\{\begin{array}{cc}
                                   f_k(p),& p\in \Sigma\setminus \left(f^{-1}_k(B_\sigma(x_0))\cap \tilde{D}_{k,i}\right)\\
                                   w_{k,i,\ell}(p),& p\in \left(B_\sigma(x_0)\cap \tilde{L}_{m,i}\right)
                                  \end{array}\right.
\end{equation*}
Since $\tilde{D}_{k,i}$ is topologically a disc, $\Sigma$ and $\Sigma_{k,\sigma}$ are topologically equivalent.
$f^{graph}_{k,\sigma}$ does not have the same area or enclosed volume as $f_k$, which we like to correct now.
Therefore we need some estimates on these properties. We start with the area. Here we use \eqref{eq:5_8} and Lemma \ref{A_3} to obtain
\begin{align}
\label{eq:5_26}
\begin{split}
|\operatorname{Area}(f_k)-\operatorname{Area}(f^{graph}_{k,\sigma})|\leq &2\max\{\operatorname{Area}(f_k|_{\tilde{D}_{k,i}}), \operatorname{Area}(w_{k,i,\ell})\}\\
\leq & 2\max\left\{\tilde{\mu}_{k,i,\ell}(B_\sigma(x_0)),C\int_{B_\sigma(x_0)\cap \tilde{L}_{k,i,\ell}}\, d\mathcal{H}^2\right\}\\
\leq & C\sigma^2\leq C\rho^2.
\end{split}
\end{align}
Let us proceed with the enclosed volume. Since $f_k(\Sigma)\subset B_N(0)$, for $N>0$ big enough and independent of $k$, we also get
$f^{graph}_{k,\sigma}(\Sigma_{k,\sigma})\subset B_N(0)$. 
Hence the definition of the enclosed volume and the Cauchy-Schwartz inequality yield
\begin{align}
\label{eq:5_27}
\begin{split}
 |\operatorname{Vol}(f_k) - \operatorname{Vol}(f^{graph}_{k,\sigma})| \leq & 2\max\{|\operatorname{Vol}(f_k|_{\tilde{D}_{k,i}})|, |\operatorname{Vol}(w_{k,i,\ell})|\}\\
 \leq & C(N)\max\{\operatorname{Area}(f_k|_{\tilde{D}_{k,i}}), \operatorname{Area}(w_{k,i,\ell})\}\\
 \leq & C\sigma^2\leq C\rho^2.
 \end{split}
\end{align}
We now assume we have vectorfields given as in Lemma \ref{4_2}.
Hence we find an $r>0$ and points $\eta_X,\eta_Y\in \operatorname{spt}(\mu)\setminus B_{r}(x_0)$.
Without loss of generality we assume $\rho_0,\rho_{x_0}\leq \frac{r}{2}$.
Furthermore for every $\varepsilon_X,\varepsilon_Y>0$ we find
vectorfields $X\in C^\infty_0(B_{\varepsilon_X}(\eta_X))$ and $Y\in C^\infty_0(B_{\varepsilon_Y}(\eta_Y))$ such that
 \begin{align}
 \begin{split}
  \label{eq:5_19}
  &\int \operatorname{div}_{T_x\mu} X(x)\, d\mu(x) \int \langle Y,\sigma_R\rangle\, d|\triangledown\theta_R|\\
  -& \int \operatorname{div}_{T_x\mu} Y(x)\, d\mu(x) \int \langle X,\sigma_R\rangle\, d|\triangledown\theta_R|\neq0. 
  \end{split}
 \end{align}
 Here $R$ denotes the $3$-current defined in the beginning of section \ref{sec:4} and $\theta_R$ the corresponding BV-function.
 We also fixate $\varepsilon_X=\varepsilon_Y=\rho$.
Furthermore let $\Phi$ be defined as in \eqref{eq:4_1_1} and $F$ and $F_k$ be as in \eqref{eq:4_2} but with respect to $f^{graph}_{k,\sigma}$ instead.
The results of section \ref{sec:4} are still valid, since the diffeomorphism $\Phi$ does not influence the graphical comparison function (cf. Figure \ref{fig_1}).
As in the beginning of section \ref{sec:4} we obtain
\begin{equation*}
 F_k(s,t)\rightarrow F(s,t).
\end{equation*}
The oriented varifold convergence and the $L^1$ convergence of $\theta_{R_k}$ yield
\begin{equation}
 \label{eq:5_28}
 |\det DF_k(0,0)|\rightarrow |\det DF(0,0)|\geq c_0>0
\end{equation}
for a fixated constant $c_0>0$. Hence for $k$ big enough we get
\begin{equation}
 \label{eq:5_29}
 |\det DF_k(0,0)|\geq \frac{c_0}{2}.
\end{equation}
Hence $DF_k(0,0)$ is invertible.
Furthermore the mean value theorem and Lemma \ref{4_4} yield for every $T_0>0$ a $C=C(T_0)>0$, such that for every $(s,t),(s',t')\in \overline{B_T(0)}$ we have
\begin{equation}
 \label{eq:5_30}
 |DF_k(s,t) - DF_k(s',t')|\leq C \|D^2F_k\|_{L^\infty(\overline{B_T(0)})}|(s,t)-(s',t')|\leq CT
\end{equation}
if we choose $0<T<T_0$ arbitrary.
By the formula for the inverse matrix via the adjunct matrix we obtain for $k$ big enough
\begin{equation}
 \label{eq:5_31}
 |(DF_k(0,0))^{-1}|\leq C
\end{equation}
and the constant is independent of $k$.
Next we will apply Lemma \ref{A_2}. Hence we need to define a functions $\tilde{F}$ and $\tilde{F}_k$ satisfying the assumptions of that Lemma:
\begin{equation}
 \label{eq:5_32}
 \begin{array}{ccc}
 \tilde{F}(s,t)&:=& (DF(0,0))^{-1} \left(F(s,t) - F(0,0)\right)\\
 \tilde{F}_k(s,t)&:=& (DF_k(0,0))^{-1} \left(F_k(s,t) - F_k(0,0)\right)
 \end{array}
\end{equation}
Here ${}^{-1}$ is meant as the matrix inverse.
Hence
\begin{equation*}
\begin{array}{cc}
 \tilde{F}(0,0)=0,\quad & D\tilde{F}(0,0)=I\in\R^{2\times2}\\
 \tilde{F}_k(0,0)=0,\quad & D\tilde{F}_k(0,0)=I\in\R^{2\times2}\end{array}
\end{equation*}
Let furthermore $(s,t),(s',t')\in B_T(0)$ for $0<T<T_0$. Then by \eqref{eq:5_30} and \eqref{eq:5_31} we have
\begin{align*}
 |D\tilde{F}_k(s,t) - D\tilde{F}_k(s',t')| \leq& |(DF_k(0,0))^{-1}|\cdot|DF_k(s,t)-DF_k(s',t')|\\
 \leq& CT<CT_0=:\delta_0 < 1,
\end{align*}
if we choose $T_0$ small enough. So Lemma \ref{A_2} is applicable to $\tilde{F}_k$
and therefore we find for every $(\tilde{y},\tilde{z})\in B_{(1-\delta_0)T}(0)$ parameters $(s_k,t_k)$ with
\begin{equation*}
 \tilde{F}_k(s_k,t_k) = (\tilde{y},\tilde{z}).
\end{equation*}
By \eqref{eq:5_32} we obtain
\begin{equation*}
 F_k(s_k,t_k)= F_k(0,0) + DF_k(0,0)(\tilde{y},\tilde{z})=:(y,z). 
\end{equation*}
Since $DF_k(0,0)$ is invertible, we obtain a $\gamma>0$ (by \eqref{eq:5_29} only dependend on $T_0$), 
such that for every $(y,z)\in B_\gamma(F_k(0,0))$ we find $(s_k,t_k)\in B_{(1-\delta_0)T}(0)$ satisfying
\begin{equation*}
 F_k(s_k,t_k)=(y,z).
\end{equation*}
Furthermore we may choose $\gamma>0$ to be maximal, i.e. satisfying the following property: 
There is a $(y_0,z_0)\in \partial B_\gamma(0)$ and a $(\tilde{y}_0,\tilde{z}_0)\in \partial B_{(1-\delta_0)T}(0)$ with
\begin{equation*}
  DF_k(0,0)(\tilde{y}_0,\tilde{z}_0) = (y_0,z_0).
\end{equation*}
This yields
\begin{equation*}
 (1-\delta_0)T = |(\tilde{y}_0,\tilde{z}_0)| = |(DF_k(0,0))^{-1}(y_0,z_0)|\leq |DF_k(0,0)^{-1}||(y_0,z_0)|\leq C\gamma.
\end{equation*}
Hence
\begin{equation}
 \label{eq:5_33}
 C(1-\delta_0)T\leq \gamma.
\end{equation}
Now we choose $T_0:=\rho_0$, $T:=\rho$ (by choosing $\rho_0$ small enough our results are still true). 
For later purposes we also state that the inverse inequality of \eqref{eq:5_34} is true as well, only with a bigger constant of course, i.e. we have with $C_1< C_2$ independent of $k$
\begin{equation}
 \label{eq:5_34}
 C_1 \rho \leq \gamma \leq C_2\rho.
\end{equation}
Hence $\gamma$ can at most decay linearly in $\rho$, while \eqref{eq:5_26} and \eqref{eq:5_27} show a quadratic error in $\rho$ for the area and volume.
By choosing $\rho_0$ small enough, we therefore obtain for $k$ big enough parameters $(s_k,t_k)\in B_\gamma(F_k(0,0))$ with
\begin{equation*}
 F_k(s_k,t_k) = (Area_0,Vol_0).
\end{equation*}
Let $V_k^{graph}$ be the oriented varifold induced by $f^{graph}_{k,\sigma}$ and let us call
$V_k^{s,t}:= \Phi(s,t,\cdot)_\sharp V_k^{graph}$. The corresponding mass is called $\mu_k^{s,t}$. 
We denote with $A_k^{s,t}$ the second fundamentalform and with $H_k^{s,t}$ the mean curvature vector of $V_k^{s,t}$.
The orientation is called $\xi_k^{s,t}$.


Since we cannot replace $v_{k,i,\ell}$ by $w_{k,i,\ell}$ and still have a minimising sequence, we will have to correct the resulting error by the diffeomorphism $\Phi$.
By the mean value theorem and Lemma \ref{4_5} the $L^2$-Norm of the second fundamental form of the $\Phi$-corrected varifolds are controlled:
\begin{equation}
 \label{eq:5_35}
 \left|\int |A_k^{s,t}|\, d\mu_k^{s,t} - \int |A_k^{0,0}|\, d\mu_k^{s,t}\right| \leq C|(s,t)| \leq C\rho.
\end{equation}
Since $F_k$ is continuously differentiable and the derivative is bounded independently of $k$ (see the proof of Lemma \ref{4_4}), the area is controlled as well:
\begin{equation}
 \label{eq:5_36}
 \left|\mu_k^{s,t}(\R^3) - \mu_k^{0,0}(\R^3)\right| = |F_k(s,t) - F_k(0,0)|\leq C|(s,t)|\leq C\rho. 
\end{equation}
The Helfrich energy is controlled by Remark \ref{4_6} and again the mean value theorem
\begin{align}
 \label{eq:5_37}
 \begin{split}
 &\bigg|\int_{\supp(\Phi)\times G^0(2,3)}|H_k^{s,t}(x) - H_0(*\xi)|^2\,dV_k^{s,t}(x,\xi)\\
 -& \int_{\supp(\Phi)\times G^0(2,3)}|H_k^{0,0}(x) - H_0(*\xi)|^2\,dV_k^{0,0}(x,\xi)\bigg|\leq C\rho.
\end{split}
 \end{align}
Let us denote with $A^w_{k,i,\ell}$ the second fundamental form, $H^w_{k,i,\ell}$ the mean curvature vector, 
$\xi_{w_{k,i,\ell}}$ the orientation and with $K^w_{k,i,\ell}$ the Gauss curvature of $\operatorname{graph}(w_{k,i,\ell})$.
By the Gauss-Bonnet Theorem $\int_\Sigma K_k\, d\mu_{g_k}$ is given entirely by the topology of $\Sigma$. 
Hence $f_{k}$ is also a minimising sequence for $f\mapsto W_{H_0,\lambda}(f)+\kappa \int_\Sigma K_f\, d\mu_g$, $\kappa\in\R$ arbitrary under prescribed area and enclosed volume. 
Here $K_f$ denotes the Gauss curvature of a given immersion $f:\Sigma\rightarrow\R^3$. 
By \cite[Eq. (11)]{DeckGruRoe} we have $|A_k|^2 = |H_k|^2 - 2K_k$ and $K_k$ is the Gauss curvature of $f_k$.
The following calculation is an adaptation to our situation from \cite[pp. 10-11]{EichmannHelfrichBoundary}:
{\allowdisplaybreaks
\begin{align*}
 &\int_{B_{\sigma}(x_0)}|A_k|^2\,d\tilde{\mu}_{k,i,\ell} \\
 =&  \int_{B_{\sigma}(x_0)}|H_k|^2\,d\tilde{\mu}_{k,i,\ell} - 2 \int_{B_{\sigma}(x_0)}K_k\,d\tilde{\mu}_{k,i,\ell}\\
 \overset{\eqref{eq:2_6_1}}{\leq}& C \bigg(\int_{B_{\sigma}(x_0)}|H_k-H_0(*\xi_{f_{k}})|^2\,d\tilde{\mu}_{k,i,\ell}+ \tilde{\mu}_{k,i,\ell}(B_\sigma(x_0)) \bigg)
 \\&-  2\int_{B_{\sigma}(x_0)}K_k\,d\tilde{\mu}_{k,i,\ell}\\
 \leq& C\left(\int_{B_{\sigma}(x_0)}|H_k-H_0(*\xi_{f_{k}})|^2\,d\tilde{\mu}_{k,i,\ell} -  \frac{2}{C}\int_{B_\sigma(x_0)}K_k\,d\tilde{\mu}_{k,i,\ell}\right)\\
 & + C\tilde{\mu}_{k,i,\ell}(B_\sigma(x_0)) \\
 \leq& C\bigg(\int_{\operatorname{graph}(w_{k,i,\ell})}|H^w_{k,i,\ell}-H_0(*\xi_{w_{k,i,\ell}})|^2\, d\mathcal{H}^2- \frac{2}{C}\int_{\operatorname{graph}(w_{k,i,\ell})}K^w_{k,i,\ell}\,d\mathcal{H}^2 \bigg)\\
 &+ \bigg|\int_{\supp(\Phi)\times G^0(2,3)}|H^{s_k,t_k}_{k}-H_0(*\xi)|^2\, dV_k^{s_k,t_k} - \frac{2}{C}\int_{\supp(\Phi)}K^{s_k,t_k}_k\,d\mu_k^{s_k,t_k}\\
 & - \int_{\supp(\Phi)\times G^0(2,3)}|H^{0,0}_{k}-H_0(*\xi)|^2\, dV_k^{0,0} + \frac{2}{C}\int_{\supp(\Phi)}K^{0,0}_k\,d\mu_k^{0,0}\bigg|\\
 &+ C\tilde{\mu}_{k,i,\ell}(B_\rho(x_0))  +\varepsilon_k\\
 \leq& C\bigg(\int_{\operatorname{graph}(w_{k,i,\ell})}|H^w_{k,i,\ell}|^2\, d\mathcal{H}^2 + \mathcal{H}^2(\operatorname{graph}(w_{k,i,\ell}))\\
 & + \frac{2}{C}\int_{\operatorname{graph}(w_{k,i,\ell})}K^w_{k,i,\ell}\,d\mathcal{H}^2\bigg)+ C\tilde{\mu}_{k,i,\ell}(B_\rho(x_0)) + C\rho +\varepsilon_k\\
 \overset{\ref{A_3}}{\leq}& C\left(\int_{B_\sigma(x_0)\cap \tilde{L}_{k,i,\ell}}|D^2w_{k,i,\ell}|^2\, dx  +\sigma\right) + C\tilde{\mu}_{k,i,\ell}(B_\rho(x_0)) + C\rho +\varepsilon_k\\
 \overset{\ref{A_3},\eqref{eq:5_8}}{\leq}& C\sigma \int_{\graph\left(v_{k,i,\ell}|_{\partial B_\sigma(x_0)\cap \tilde{L}_{k,i,\ell}}\right)}|A_{m,i}|^2d\mathcal{H}^1 + C\sigma + C\rho+\varepsilon_k
\end{align*}}
Here $\varepsilon_k\rightarrow 0$ for $k\rightarrow \infty$.
The estimates connecting $D^2w_{k,i,\ell}$ with the corresponding curvatures can be seen by e.g. \cite[Subsection 2.1]{DeckGruRoe} and the bound on the gradient of $w_{k,i,\ell}$.
Integrating over $S_\rho$ together with the Co-Area formula (see e.g. \cite[Eq. (10.6)]{Simon_Buch}) yields
\begin{equation*}
 \int_{B_\frac{\rho}{2}(x_0)}|A_k|^2\, d\tilde{\mu}_{k,i,\ell} \leq C\int_{B_{\frac{3}{4}\rho}(x_0)\setminus B_\frac{\rho}{2}(x_0)}|A_k|^2\, d\tilde{\mu}_{k,i,\ell} + C\rho + \varepsilon_k.
\end{equation*}
Summing over $\ell$ yields with the help of \eqref{eq:5_24}
\begin{equation*}
 \int_{B_\frac{\rho}{2}(x_0)}|A_k|^2\, d\mu_{k,i} \leq C\int_{B_{\frac{3}{4}\rho}(x_0)\setminus B_\frac{\rho}{2}(x_0)}|A_k|^2\, d\mu_{k,i} + C\rho + \varepsilon_k.
\end{equation*}
Hole filling yields a $0<\Theta<1$ independent of $k$ satisfying
\begin{equation*}
 \int_{B_\frac{\rho}{2}(x_0)}|A_k|^2\,d\mu_{k,i} \leq \Theta \int_{B_{\frac{3}{4}\rho}(x_0)}|A_k|^2\, d\mu_{k,i} + C\rho + \varepsilon_k. 
\end{equation*}
The semi-continuity properties of $\nu$ (see e.g. \cite[Prop. 4.26]{Maggi}) for measure convergence yield for $k\rightarrow\infty$
\begin{equation}
\label{eq:5_38}
 \nu_i(B_{\frac{\rho}{2}}(x_0)) \leq \Theta \nu_i(B_\rho(x_0)) + C\rho.
\end{equation}
Here $|A_k|^2\mu_{k,i}\rightarrow \nu_i$ for $k\rightarrow\infty$ as Radon measures.
By $\mu_{k,i}\leq\mu_{k}$ we also get $\nu_i\leq \nu$.
Since we only needed the estimate
\begin{equation*}
 \nu(\overline{B_{\rho_{x_0}}(x_0))} \leq \varepsilon_0^2
\end{equation*}
to obtain \eqref{eq:5_38}, we can repeat the argument for every $B_\rho(x)\subset B_{\frac{\rho_{x_0}}{4}}(x_0)$.
This yields for these balls
\begin{equation*}
 \nu_i(B_{\frac{\rho}{2}}(x_0)) \leq \Theta \nu_i(B_\rho(x)) + C\rho.
\end{equation*}
E.g. \cite[Lemma 8.23]{GilbargTrudinger} gives us
\begin{equation}
 \label{eq:5_39}
 \nu_i(B_\rho(x)) \leq C\left(\frac{\rho}{\rho_{x_0}}\right)^\beta\nu_i(B_\frac{\rho_{x_0}}{4}(x)) + C\sqrt{\rho \rho_{x_0}}.
\end{equation}
Here $C=C(\Theta)>0$ and $\beta=\beta(\Theta)>0$ are constants.
Since $B_\frac{\rho_{x_0}}{4}(x)\subset B_{\rho_{x_0}}(x_0)$ we also get
\begin{equation}
 \label{eq:5_40}
 \nu_i(B_\rho(x)) \leq C\left(\frac{\rho}{\rho_{x_0}}\right)^\beta\nu_i(B_{\rho_{x_0}}(x_0)) + C\sqrt{\rho \rho_{x_0}}.
\end{equation}
Since $\frac{\rho}{\rho_{x_0}}< 1$ we can choose $\beta < \frac{1}{2}$ and obtain
\begin{align*}
 \nu_i(B_\rho(x))\leq& C\rho^\beta\rho_{x_0}^{-\beta}\left(\nu_i(B_{\rho_{x_0}}(x_0)) + \rho^{\frac{1}{2}-\beta} \rho_{x_0}^{\frac{1}{2}+\beta}\right)\\
 \leq & C\rho^\beta\rho_{x_0}^{-\beta}\left(\nu(\overline{B_{\rho_{0}}(x_0)}) + \rho_0\right)\\
 \leq & C\rho^\beta\rho_{x_0}^{-\beta}\left(\varepsilon_0^2 + \rho_0\right).
\end{align*}
By $|A_{\mu_i}|^2\mu_i \leq \nu_i$ we therefore get
\begin{equation}
 \label{eq:5_41}
 \int_{B_\rho(x)}|A_{\mu_{i}}|^2\, d\mu_i \leq C\rho^\beta\rho_{x_0}^{-\beta}\left(\varepsilon_0^2 + \rho_0\right).
\end{equation}
Choosing $C\left(\varepsilon_0 + \rho_0\right)$ small enough, Allard's regularity theorem \ref{A_1} (cf. \eqref{eq:A_2}) yields $\mu_i$ to be a $C^{1,\beta}$ graph.
By \eqref{eq:5_12} $\mu$ is a union of $C^{1,\beta}$ graphs in a neighbourhood of $x_0$, which all satisfy estimates in the form of \eqref{eq:A_5}.
This finishes the proof.
\end{proof}

\begin{remark}
 \label{5_2}
 The proof of Lemma \ref{5_1} does not work if we would minimise in the class of embeddings, 
 since the lack of a Li-Yau-type inequality prevents us from showing that the $f^{graph}_{k,\sigma}$ are still embeddings.
 This is a key problem, because we cannot compare the Helfrich energy of $f^{graph}_{k,\sigma}$ to $f_k$ without it.
\end{remark}

Next we formulate the lower-semicontinuity property of the minimising sequence:
\begin{lemma}
 \label{5_3}
 The minimising sequence $V^0_k$ for the Helfrich problem \eqref{eq:1_6} satisfies
 \begin{equation*}
  W_{H_0}(V^0)\leq \liminf_{k\rightarrow\infty} W_{H_0}(V^0_k).
 \end{equation*}
\begin{proof}
 Since $\mu$ is locally a graph of $C^{1,\beta}\cap W^{2,2}$ graphs outside of finitely many points (see Lemma \ref{5_1}),
 and these graphs are approximated by $\mu_{k,i}$, see \eqref{eq:5_11},
 the proof of the lower-semicontinuity estimate is the same as in \cite[Lemma 4.1]{EichmannHelfrichBoundary}.
\end{proof}

\end{lemma}

\appendix
\section{Auxilliary Results}
\label{sec:A}


For the readers convenience we collect a few needed results:

The following  is a variant of Allard's regularity Theorem. 
A proof of this statement can be found in \cite[Section 3]{Simon} or \cite[Korollar 20.3]{Schaetzle_Skript_Allard} (see also \cite[Theorem B.1]{Schaetzle}).
\begin{theorem}[Allard's regularity Theorem, see \cite{Allard}, Theorem 8.16]
\label{A_1}
For $n,m\in\N$, $0<\beta<1$, $\alpha>0$ there exist $\varepsilon_0=\varepsilon_0(n,m,\alpha,\beta)>0$, $\gamma=\gamma(n,m,\alpha,\beta)$ and $C=C(n,m,\alpha,\beta)$ such that:\\
Let $\mu$ be an integral $n$-varifold in $B_{\rho_0}^{n+m}(0)$, $0<\rho_0<\infty$, $0<\varepsilon<\varepsilon_0$ with locally bounded first variation in $B_{\rho_0}^{n+m}(0)$ satisfying
\begin{equation}
 \label{eq:A_1}
 \rho^{1-n}\|\delta \mu\|(B_\rho)\leq \varepsilon^2(\rho^{-n}\mu(B_{\rho}))^{1-\alpha}\rho^{2\beta}\rho_0^{-2\beta},\quad \forall B_{\rho}\subset B_{\rho_0}(0)
\end{equation}
or weak mean curvature $H_{\mu}\in L^2(\mu\lfloor B_{\rho_0}^{n+m}(0))$ satisfying
\begin{equation}
 \label{eq:A_2}
 (\rho^{2-n})\left(\int_{B_\rho}|H_\mu|^2\, d\mu\right)^{\frac{1}{2}}\leq \varepsilon(\rho^{-n}\mu(B_\rho))^{\frac{1}{2}-\alpha}\rho^\beta\rho_0^{-\beta},\quad \forall B_\rho\subset B_{\rho_0}(0)
\end{equation}
and 
\begin{equation}
 \label{eq:A_3}
 0\in\operatorname{spt}\mu,\ \rho_0^{-n}\mu(B_{\rho_0}(0))\leq(1+\varepsilon)\omega_n.
\end{equation}
Then there exists $u\in C^{1,\beta}(B_{\gamma\varepsilon\rho_0}^n(0),\R^m)$ $u(0)=0$, such that after rotation
\begin{equation}
 \label{eq:A_4}
 \mu\lfloor B_{\gamma\varepsilon\rho_0}^{n+m}(0)=\mathcal{H}^n\lfloor(\operatorname{graph} u\cap B_{\gamma\varepsilon\rho_0}^{n+m}(0))
\end{equation}
and
\begin{equation}
 \label{eq:A_5}
 (\varepsilon\rho_0)^{-1}\|u\|_{L^{\infty}(B^n_{\gamma\varepsilon\rho_0}(0))}+\|\nabla u\|_{L^{\infty}(B^n_{\gamma\varepsilon\rho_0}(0))} + (\varepsilon\rho_0)^\beta \operatorname{h\ddot{o}l}_{B_{\gamma\varepsilon\rho_0}^n(0),\beta}\nabla u \leq C\varepsilon^{\frac{1}{2(n+1)}}.
\end{equation}
\end{theorem}

In section \ref{sec:3} we need a version of the inverse function theorem with explicit estimates on the size of domain and codomain on which the function is invertible:
\begin{theorem}[See \cite{Lang}, Chapter XIV §1, Lemma 1.3]
 \label{A_2}
 Let $0\in U\subset\R^n$ be open and $f\in C^1(U,\R^n)$. Furthermore let $f(0)=0$, $Df(0)=id$. 
 Assume $r>0$ with $\overline{B_r(0)}\subset U$ and let $0<s<1$ satisfy
 \begin{equation*}
  \|Df(z)-Df(x)\|\leq s
 \end{equation*}
for all $x,z\in \overline{B_r(0)}$. Here $\|Df(z)\|=\sup_{|x|=1}|Df(z)x|$. If $y\in \R^n$ and $|y|\leq (1-s)r$, then there exists a unique $x\in \overline{B_r(0)}$, such that $f(x)=y$.
\end{theorem}

The following lemma provides a suitable comparison function in section \ref{sec:4}. 
It is a generalisation by Schygulla of the biharmonic comparison principle by Simon (see \cite[Lemma 2.2]{Simon}):
\begin{lemma}[See \cite{Schygulla}, p. 938 Lemma 8]
\label{A_3}
 Let $L\subset \R^3$ be a $2$-dimensional plane, $x_0\in L$ and $u\in C^\infty(U,L^\perp)$, where $U\subset L$ is an open neighbourhood of $L\cap \partial B_\rho(x_0)$.
 Moreover let $|Du|\leq c$ on $u$. Then there exists a function $w\in C^\infty(\overline{B_\rho(x_0)}, L^\perp)$ such that
 \begin{equation*}
  w=u,\ \frac{\partial w}{\partial \nu}=\frac{\partial u}{\partial \nu}\mbox{ on }\partial B_\rho(x_0),
  \end{equation*}
  \begin{equation*}
  \frac{1}{\rho}\|w\|_{L^\infty(B_\rho(x_0))}\leq c\left(\frac{1}{\rho}\|u\|_{L^\infty(\partial B_\rho(x_0))} + \|Du\|_{L^\infty(\partial B_\rho(x_0))}\right),
  \end{equation*}
  \begin{equation*}
  \|Dw\|_{L^\infty(B_\rho(x_0))}\leq c\|Du\|_{L^\infty(\partial B_\rho(x_0))},
  \end{equation*}
   \begin{equation*}
  \int_{B_\rho(x_0)}|D^2 w(x)|^2dx\leq c\rho \int_{\graph u_{\partial B_\rho(x_0)}}|A|^2d\mathcal{H}^1.
 \end{equation*}
Here $A$ denotes the second fundamental form of $\graph u$. Furthmore $\nu$ is the outer normal of $L\cap B_\rho(x_0)$ with respect to $L$.
\end{lemma}


\phantomsection
\addcontentsline{toc}{section}{Literatur}
\bibliography{bibliography}
\bibliographystyle{plain}


\end{document}